\documentclass[11pt]{amsart}

\usepackage{amsmath}
\usepackage{amsthm}
\usepackage{hyperref}
\usepackage{xcolor}
\hypersetup{
	colorlinks,
	linkcolor={blue!80!black},
	citecolor={green!60!black},
	urlcolor={blue!80!black}
}
\usepackage[margin=1.0in]{geometry}
\usepackage[T1]{fontenc}
\usepackage{MnSymbol}

\numberwithin{equation}{section}

\newtheorem{prop}{Proposition}
\newtheorem{lemma}[prop]{Lemma}

\newtheorem{thm}[prop]{Theorem}
\newtheorem{cor}[prop]{Corollary}

\numberwithin{prop}{section}

\theoremstyle{definition}
\newtheorem{defn}[prop]{Definition}
\newtheorem{ex}[prop]{Example}
\newtheorem{rmk}[prop]{Remark}

\newtheorem{assumption}[prop]{Assumption}

\DeclareSymbolFont{script}{U}{eus}{m}{n}
\DeclareSymbolFontAlphabet{\mathscr}{script}
\DeclareMathSymbol{\Wedge}{0}{script}{"5E}
\DeclareMathAlphabet{\mathrmsl}{OT1}{cmr}{m}{sl}

\renewcommand{\div}{\mathrm{div}}
\newcommand{\vol}{\mathrm{vol}}
\renewcommand{\i}{\sqrt{-1}}
\newcommand{\gD}{\Delta}
\newcommand{\gs}{\sigma}
\newcommand{\mf}{\mathfrak}
\newcommand{\mc}{\mathcal}
\newcommand{\del}{\partial}
\newcommand{\brs}[1]{\left|#1\right|}

\newcommand{\delb}{\bar{\partial}}
\newcommand{\dt}{\frac{\partial}{\partial t}}
\newcommand{\poiss}{{\mathrm \pi}}
\newcommand{\N}{\nabla}
\newcommand{\GK}{\mathcal{GK}}

\newcommand{\II}{{\mathbb I}}
\newcommand{\JJ}{{\mathbb J}}
\newcommand{\bxi}{\boldsymbol{\xi}}
\newcommand{\til}[1]{\widetilde{#1}}
\renewcommand{\bar}[1]{\overline{#1}}

\newcommand{\J}{\mathbb J}
\newcommand{\R}{\mathbb R}
\newcommand{\C}{\mathbb C}

\newcommand{\G}{\mathbb G}

\newcommand{\IP}[1]{\left<#1\right>}

\newcommand{\Gscal}{\mathrm{Gscal}}

\newcommand{\T}{{\mathbb T}}

\newcommand{\Aut}{{\mathrm {Aut}}}
\newcommand{\hred}{\mathfrak{h}_{\mathrm{red}}}
\newcommand{\mab}[1]{\llangle #1 \rrangle}

\newcommand{\la}{\langle}
\newcommand{\ra}{\rangle}
\newcommand{\ga}{\alpha}
\newcommand{\gb}{\beta}
\newcommand{\gw}{\omega}

\DeclareMathOperator{\Real}{Re}
\DeclareMathOperator{\Rc}{Rc}
\DeclareMathOperator{\ad}{ad}

\DeclareMathOperator{\tr}{tr}
\DeclareMathOperator{\Ker}{Ker}

\DeclareMathOperator{\Id}{Id}

\DeclareMathOperator{\im}{Im}

\DeclareMathOperator{\End}{End}

\DeclareMathOperator{\Diff}{Diff}

\newcommand{\XX}{\mathbf{X}}

\title{Formal structure of scalar curvature in generalized K\"ahler geometry}

\author{Vestislav Apostolov}
\address{V.\,Apostolov\\ D{\'e}partement de Math{\'e}matiques\\ UQAM \\
 and \\ Institute of Mathematics and Informatics\\ Bulgarian Academy of Sciences}
\email{\href{mailto:apostolov.vestislav@uqam.ca}{apostolov.vestislav@uqam.ca}}

\author{Jeffrey Streets}
\address{J.\,Streets \\ Rowland Hall\\
        University of California\\
        Irvine, CA 92617}
\email{\href{mailto:jstreets@uci.edu}{jstreets@uci.edu}}

\author{Yury Ustinovskiy}
\address{Y.\,Ustinovskiy, Susquehanna International Group}
\email{\href{mailto:yura.ust@gmail.com}{yura.ust@gmail.com}}

\date{April 22, 2024}

\begin{document}

\begin{abstract} Building on works of Boulanger \cite{Boulanger} and Goto \cite{Goto_2020,goto-21}, we show that Goto's scalar curvature is the moment map for an action of generalized Hamiltonian automorphisms of the associated Courant algebroid, constrained by the choice of an adapted volume form.  We derive an explicit formula for Goto's scalar curvature, and show that it is constant for generalized K\"ahler-Ricci solitons.  Restricting to the generically symplectic type case, we realize the generalized K\"ahler class as the complexified orbit of the Hamiltonian action above.  This leads to a natural extension of Mabuchi's metric and $K$-energy, implying a conditional uniqueness result.  Finally, in this setting we derive a Calabi-Matsushima-Lichnerowicz obstruction and a Futaki invariant.
\end{abstract}

\maketitle

\section{Introduction}

The classical uniformization of Riemann surfaces is a central result in complex analysis, establishing a fundamental link between metric geometry, complex geometry, and topology.  Furthermore it provides the foundation for deeper studies of moduli spaces of Riemann surfaces.  This definitive theory forms the archetype for uniformization questions in higher dimensions, where linkages are sought between metric and complex geometry.  In \cite{CalabiExt} Calabi initiated an expansive program of finding canonical representatives of K\"ahler metrics in a fixed deRham class.  He proposed extremal K\"ahler metrics, i.e. metrics whose scalar curvature defines a holomorphic Hamiltonian vector field, as the candidates for such canonical metrics.  This problem combines the well-known existence problems for constant scalar curvature and K\"ahler-Einstein metrics, and is guided by the central Yau-Tian-Donaldson (YTD) conjecture, which asserts that the obstruction for existence of an extremal K\"ahler metric is expressed using an algebraic notion of stability for the K\"ahler class.  In recent years the subject of generalized K\"ahler geometry \cite{GHR, GualtieriThesis, GualtieriGKG, HitchinGCY} has emerged as a deeply structured extension of K\"ahler geometry, which captures further natural interactions between complex, Poisson, and symplectic geometry.  Building on seminal work of Goto \cite{Goto_2020, goto-21}, in \cite{ASUScal} the authors gave an extension of analytic aspects of Calabi's program to the setting of generalized K\"ahler structures of symplectic type.  In this work we extend this further to arbitrary generalized K\"ahler structures.  This general setup can be applied to underlying complex Poisson manifolds which do not admit symplectic structures at all, such as Hopf surfaces.

\subsection{Scalar curvature as moment map}

We recall briefly that generalized K\"ahler structures admit two descriptions on the background of $(M, H_0)$,  a smooth manifold $M$ endowed with a closed three-form $H_0$.  The first is described in terms of bihermitian geometry: a generalized K\"ahler structure is a quadruple $(g, I, J, b)$ of a Riemannian metric $g$, a two-form $b$ and two integrable almost complex structures $I, J$ such that $g$ is compatible with  $I$ and $J$ and further satisfy
\begin{align*}
    - d^c_I \gw_I = H_0 + db = d^c_J \gw_J,
\end{align*}
where $\gw_I = gI, \gw_J = gJ$ and $d^c_I = \i (\delb_I - \del_I), \, d^c_J= \i (\delb_J - \del_J)$.  The second description is in terms of generalized complex geometry: a generalized K\"ahler structure is a pair of commuting generalized complex structures $\II, \JJ$ on the exact Courant algebroid $(TM \oplus T^*M, H_0)$ further satisfying $\IP{- \II\JJ,} > 0$, where $\IP{,}$ is the canonical neutral inner product.  These two descriptions are equivalent by the work of Gualtieri \cite{GualtieriThesis} (cf. Theorem \ref{thm:gualtieri_map} below).  There are several natural Poisson tensors associated to a generalized K\"ahler structure.  First, as observed by Hitchin \cite{HitchinPoisson}, the tensor
\begin{align}\label{eq:Poisson_BH}
    \pi := \tfrac{1}{2} [I, J] g^{-1}
\end{align}
defines a real Poisson structure.  Furthermore, for any generalized complex structure $\JJ$ one has a canonically associated real Poisson tensor
$\pi_{\JJ} (\xi, \eta) := \IP{\JJ \xi, \eta}$.
Thus, a generalized K\"ahler structure comes equipped with three (generally distinct) real Poisson tensors: $\pi$ as above, together with
\begin{align*}
    \pi_{\II} = \tfrac{1}{2}(I - J) g^{-1}, \qquad \pi_{\JJ} = \tfrac{1}{2} (I + J) g^{-1}.
\end{align*}
In the special case when the generalized K\"ahler structure is of symplectic type, the endomorphism $I + J$ is invertible, $F = -2 g (I + J)^{-1}$ is a symplectic form, and  $\pi_{\JJ} = F^{-1}$.  With this in mind the results to follow are all generalizations of our prior work \cite{ASUScal} in the symplectic type case.

Viewed from the bihermitian description, one notes that in general there is no single connection which preserves all of the data, and thus it is not clear at the outset what the relevant curvature quantities are.  In view of this, works of Boulanger \cite{Boulanger} and Goto \cite{goto-21,Goto_2020} sought to define a relevant scalar curvature by generalizing the Fujiki-Donaldson work in K\"ahler geometry realizing the scalar curvature as the moment map for a natural action of the group of Hamiltonian diffeomorphisms for a fixed symplectic form.   This culminated in \cite{goto-21}, where Goto identified a natural Lie algebra of infinitesimal automorphisms of $(TM \oplus T^*M, H_0)$ generated by elements $\JJ du$, where $u$ is a smooth function on $M$ preserved by the divergence of $\JJ$ relative to a fixed, \emph{arbitrary} volume form (cf. \S \ref{ss:GCdiv}).  These act on the space of `almost generalized K\"ahler' structures compatible with $\JJ$, a space which inherits a natural K\"ahler structure.  Furthermore, Goto defines a scalar curvature quantity associated to a generalized K\"ahler structure $(\II, \JJ)$ and a volume form, and relates the variation of this scalar curvature to deformations by generalized Hamiltonians described above.  Given this variational formula, certain `modified Hamiltonians' are introduced and it is claimed that the scalar curvature is the moment map for the action by these.

This delicate and surprising work suggests many further questions.  First, it appears that the `modified Hamiltonians' described in \cite{goto-21} do not form a Lie algebra.  For this reason we present in \S \ref{s:scalasmoment} a slight modification of the setup in \cite{goto-21},  which places a further restriction on the generalized K\"ahler structures.  As described above, there is a canonical infinitesimal symmetry of $(TM \oplus T^*M, H_0)$ associated to $\JJ$ and a choice of volume form $\vol$ on $M$, denoted $\div_{\vol}(\JJ)$.  Inspired by recent work of Inoue~\cite{Inoue} in the K\"ahler case, we ask for a volume form $\vol$ such that $\div_{\vol}(\JJ)$ (which always preserves $\JJ$) preserves $\II$.  In this case, we call the volume form \emph{adapted}.  In terms of the bihermitian description $(g, I, J, b)$, expressing $\vol = e^{-f}dV_g$, if  $\vol$ is adapted then the vector field
\[ X= \frac{1}{2}\left(d^*_g(\omega_I + \omega_J) + (I+ J)\nabla^g f\right)\]
is Killing for $(g, I, J)$.  A subtle issue is that the condition that $X$ preserves the bihermitian data does not seem to imply that the volume form is adapted, although we show that this is true in the generically symplectic type case (cf. Lemma \ref{l:adapted_converse}).  We let $\T$ be the torus of isometries generated by the Killing vector field $X$.  In the K\"ahler case, when we have $I=J$ and using $\T$-invariant potentials, we recover the formal moment map setting considered by Inoue in \cite{Inoue}, leading to the so-called \emph{$\mu$-scalar curvature 
K\"ahler metrics} and extending the notion of K\"ahler-Ricci solitons on Fano manifolds~\cite{Tian-Zhu} to arbitrary K\"ahler classes.

Considering the space $\mc{AGK}(\JJ, \vol)$ of all $\II$-compatible and $\div_{\vol} (\JJ)$-invariant almost generalized K\"ahler structures $(\II, \JJ)$ on $M$,  Goto's original variational formulas lead directly to the claim that the generalized scalar curvature is the moment map for the infinitesimal action of $\T$-invariant Hamiltonians $\mf{ham}^{\T}(\JJ)$ on the space $\mc{AGK}(\JJ, \vol)$: (cf. \S \ref{s:scalasmoment} for precise notation) 
\begin{thm} \cite{goto-21} (Theorem \ref{t:gotomoment}) \label{t:gotomomentintro}
The natural action of $\mf{ham}^\T(\JJ)$ on $\mc{AGK}(\JJ, \vol)$ is Hamiltonian with moment map given by the generalized scalar curvature.  More concretely,
\begin{equation*}
\pmb{\Omega} \left(L^{H_0}_{\JJ du}\II, \left. \dt \right|_{t=0}\II_t \right)=
\int_M \left(\dt\Gscal(\II_t,\JJ,\vol)\right)u\vol
\end{equation*}
for any one-parameter family $\II_t\in \mc{AGK}(\JJ, \vol)$.
\end{thm}

\subsection{Scalar curvature and generalized K\"ahler-Ricci solitons}

Goto's scalar curvature is defined using the local spinorial description of generalized K\"ahler structures, and thus it is not clear what this quantity really is in terms of the classical bihermitian description.  By now formulas in some special cases have appeared \cite{Boulanger, Goto_2020, goto-21, ASUScal,AFSU} which hint at the general form of the answer.  Here, by exploiting the nondegenerate perturbation method introduced in \cite{AFSU}, we obtain the general answer:    
\begin{thm} (Theorem \ref{thm:scalformula}) \label{thm:scalformulaintro}
    Let $(M,g,I,J, b)$ be a generalized K\"ahler manifold determined by generalized complex structures $(\II,\JJ)$ on $(M, H_0)$.  Then for all $f \in C^{\infty}(M)$ one has
    \begin{align*}
        \Gscal(\II,\JJ,e^{-f} dV_g)=\frac{1}{4}\left(
            R-\frac{1}{12}|H|^2+2\Delta f-|df|^2
        \right),
    \end{align*}
    where $H:= - d_I^c\omega_I= d_J^c\omega_J= H_0 +db$.
\end{thm}
\noindent This quantity has already appeared in several distinct contexts, including as the integrand of the string effective action \cite{Polchinski}, in Bismut's work on the index theorem in complex geometry \cite{Bismut}, and as the natural notion of scalar curvature in generalized geometry \cite{CSCW,GF19} and generalized Ricci flow \cite{GRFbook, Streetsscalar}, making Goto's derivation from the spin geometry of generalized K\"ahler structures quite surprising.

Theorem \ref{thm:scalformulaintro} shows that the moment map framework above captures a number of interesting examples.  Generalized K\"ahler-Ricci solitons are self-similar solutions of  generalized K\"ahler-Ricci flow \cite{GKRF} and serve as the canonical geometries behind extensions of the Calabi-Yau theorem to generalized K\"ahler geometry (cf. \cite{AFSU, ASnondeg}).  By a Bianchi identity (Proposition \ref{p:Bianchi}) we show that all generalized K\"ahler-Ricci solitons have constant scalar curvature relative to the weighted volume.  We furthermore show that the vector field $X$ above is automatically Killing for a soliton, thus showing that generically symplectic type solitons fit into the GIT framework above.  In a series of works \cite{Streetssolitons, SU1, ASU3} the authors constructed and classified all generalized K\"ahler-Ricci solitons on compact complex surfaces, with non-K\"ahler generically symplectic type examples occurring on all type 1 Hopf surfaces.  In higher dimensions there are the examples of canonical generalized K\"ahler structures on semisimple Lie groups.  These were shown to have constant positive Goto scalar curvature by Goto, using the spinorial definition directly \cite{goto-21}.  The constant Goto scalar curvature of these examples follows immediately from Theorem \ref{thm:scalformulaintro}, and in the case the group has no abelian factors we construct the relevant adapted volume form and again give the realization of these geometries as zeroes of the moment map.

\subsection{Complexified orbit and Calabi problem}

A further fundamental aspect of the Fujiki-Donaldson moment map framework is the fact that on a given complex manifold, the space of K\"ahler metrics in a fixed deRham class admits a natural interpretation as the complexified orbit of the action of the group of Hamiltonian symplectomorphsims.  This point follows by Moser's lemma, and in \cite{ASUScal} we extended this to the symplectic type generalized K\"ahler case.  In \cite{GualtieriBranes} a notion of generalized K\"ahler class of structures is introduced via certain infinitesimal deformations of the bihermitian data.  We show in Lemma \ref{l:moser} that these deformations correspond formally to the complexified orbit for the natural complex structure  on $\mc{AGK}(\JJ)$ in the moment map setup we present, after pullback by an appropriately chosen family of Courant automorphisms.  However, ensuring that the complexified orbit remains adapted to the volume form is a subtle point.  In general one can construct a family of volume forms associated to a path in the generalized K\"ahler class by pullback of the fixed adapted volume by the automorphisms constructed above.  In general, this volume form will be adapted to the bihermitian data, and in case the initial GK structure is generically symplectic type, it will be adapted to $(\II, \JJ)$.  Furthermore we show in this case that the volume form, which is constructed pathwise in the nonlinear space of generalized K\"ahler structures, does not depend on the chosen path.

Continuing in the generically symplectic type case, using the canonically associated adapted volume form constructed above, we define a generalization of the Mabuchi metric \cite{Mabuchi} (cf. \cite{Semmes, Donaldson-git}), further extending the construction from \cite{ASUScal}.  This is a formal Riemannian metric on the generalized K\"ahler class.  We derive the Levi-Civita connection and geodesic equation, and we show in Proposition \ref{p:mabuchicurvature} that it is formally a symmetric space with nonpositive sectional curvature.  With these tools in place we can define an extension of the Calabi problem.  In particular we define the Calabi energy as the $L^2$ norm of the Goto scalar curvature, and show in Proposition~\ref{p:Calabivar} that critical points are precisely the extremal GK structures.  Furthermore by integrating the Goto scalar curvature against an invariant Hamiltonian potential we obtain a generalization of the Mabuchi one-form.  We show in Proposition~\ref{p:Mabuchiclosed} that this form is closed, and thus there is at least a locally defined functional whose differential is this one-form.  The existence of such a functional which is well-defined on the whole generalized K\"ahler class is a nontrivial open problem.  We furthermore show in Proposition~\ref{p:Mabuchiconvex} that the Mabuchi one-form is monotone along Mabuchi geodesics, which leads in Corollary~\ref{c:conditional-uniqueness} to the claim that any two critical points connected by a smooth geodesic are in fact isometric.

\subsection{Obstruction theory}

In the final Section 7, we  extend  in Theorem~\ref{thm:CML} the Calabi-Lichnerowicz-Matsushima theorem to general extremal GK structures. Similarly to the work of Goto~\cite{Goto_LM} in the symplectic-type case, the general version of this result relates the stabilizers respectively in $\mf{ham}^{\T}$ and in its complexification $\mf{ham}^{\T}\otimes \C$ of an extremal GK structure. This relationship follows from purely formal arguments,  using the the moment map interpretation of the Goto scalar curvature (see Appendix~\ref{s:CLM}). In order to obtain  further ramifications, expressed as the reductiveness of a certain Lie algebra of real holomorphic vector fields, we assume generically symplectic type and obtain in Theorem~\ref{thm:CML-reduced} a partial generalization of our previous results \cite{ASU2} which were obtained in the symplectic type case. We finally derive in Theorem~\ref{t:futaki-character},  as a byproduct of our definition of Mabuchi $1$-form, a version of the Futaki invariant associated to a GK class.

\vskip 0.1in
\textbf{Acknowledgements:} VA was supported by an NSERC Discovery Grant and a Connect Talent Grant of the R\'egion des Pays de la Loire. The work of JS was supported by the NSF via DMS-2203536 and a Simons Fellowship.  The authors thank Mario Garcia-Fernandez and Ryushi Goto for helpful comments.

\section{Background} \label{s:background}
\subsection{Generalized geometry}

Consider a smooth manifold $M$ of real dimension $2n$ endowed with a closed $3$-form $H_0$. The pair $(M, H_0)$ gives rise to an exact Courant algebroid $E=TM\oplus T^*M$ with neutral inner product and $H_0$-twisted Dorfman bracket defined on sections of $E$ by
\begin{align*}
\IP{X + \xi, Y + \eta} =&\ \tfrac{1}{2} \left( \eta(X) + \xi(Y) \right)\\
[X+\xi,Y+\eta]_{H_0}=&\ [X,Y]+L_X\eta-i_Yd\xi+i_Yi_XH_0.
\end{align*}
We have a natural \emph{spinorial} action of elements of $E$ on $\Wedge^*(M)$ by:
\[
(X+\xi)\cdot \alpha = i_X\alpha+\xi\wedge \alpha.
\]
It is straightforward to see that this action extends to an action of the Clifford algebra of $E=TM\oplus T^*M$ on $\Lambda^*(T^*)$, i.e. satisfies:
\[(X+\xi)\cdot\left((X+\xi)\cdot \alpha\right)=\xi(X)\alpha = \IP{X + \xi, X + \xi} \ga.
\]
Thus, $\Wedge^*(M)$ is a Clifford module.  All of the above can be complexified to define a Clifford action of $(TM\oplus T^*M)\otimes \C$ on \[\Wedge^*(M)_\C:=\Wedge^*(M)\otimes\C.\] 
This algebra is equipped with a twisted de Rham differential
\[
d_{H_0}\alpha:=d\alpha-H_0\wedge \alpha.
\]
Furthermore, there is a $\Wedge^{2n}(M)_\C$-valued $\mathrm{Spin}$-invariant \textit{Mukai pairing} 
$(\cdot,\cdot)\colon \Wedge^*(M)_\C\otimes \Wedge^*(M)_\C\to \Wedge^{2n}(M)_\C$ defined by
\[
(\alpha,\beta)_M:=[\alpha\wedge\sigma(\beta)]_{\mathrm{top}},
\]
where $\alpha,\beta\in\Wedge^*(M)_\C$; $\sigma\colon \Wedge^*(M)\to\Wedge^*(M)$ is the Clifford involution defined on the decomposable forms as $\sigma(dx^{i_1}\wedge\dots \wedge dx^{i_k}):=dx^{i_k}\wedge\dots \wedge d x^{i_1}$; and $[\alpha]_{\mathrm{top}}$ denotes the top degree component of $\alpha\in\Wedge^*(M)_\C$.
\begin{rmk}
With our convention for the normalization of the Mukai pairing,  if $\omega\in\Wedge^2(M)$ and $\varphi:=e^{\i\omega} \in \Wedge^{\rm even}(M)_{\C}$, we have
\[
(\varphi,\bar{\varphi})_M=(2\i)^n\frac{\omega^n}{n!}.
\]
\end{rmk}

Given a diffeomorphism $\Phi\in\mathrm{Diff}(M)$ we lift it to $E=T\oplus T^*$ as
\[
\left(
\begin{matrix}
    \Phi_* & 0 \\ 0 & (\Phi^*)^{-1}
\end{matrix}
\right)
\]
and,  by a standard abuse of notation,  denote the corresponding action by $\Phi$:
\[
\Phi\cdot(X+\xi):=\Phi_*(X)+(\Phi^*)^{-1}(\xi).
\]
Furthermore for $b\in\Wedge^2(M)$ we define a $\langle \cdot, \cdot \rangle$-orthogonal endomorphism $e^b\in\End(TM\oplus T^*M)$ by
\[
e^b(X+\xi):=X+\xi+b(X,\cdot).
\]
Combining the two actions,  we obtain the extended diffeomorphism group action $TM \oplus T^*M$:
\[\{
\Phi\circ e^b\ |\ \Phi\in\mathrm{Diff}(M),\ b\in\Wedge^2(M),\}\]
with product given by
\[ (f_1\circ e^{b_1})\circ (f_2\circ e^{b_2})=f_1\circ f_2\circ e^{b_2+f_2^*b_1}.\]
The corresponding Lie algebra is $C^{\infty}(M, TM \oplus \Wedge^2(M))$ with Lie bracket
\[
[X_1+b_1,X_2+b_2]_{\rm aut} =[X_1,X_2]+L_{X_1}b_2-L_{X_2}b_1.
\]
The extended diffeomorphism group leads to a Lie derivative of sections of $TM\oplus T^*M$, extending the usual Lie derivative with respect to a vector field. The extended Lie derivative is explicitly given by
\begin{equation}\label{eq:Lie-derivative}
L_{X+ b} (Y+ \eta) = [X, Y]  + L_X \eta -\imath_Y b.\end{equation}

An extended diffeomorphism $e^b$ preserves the $H_0$-twisted Dorfman  bracket if and only if $db=0$. In general, it transforms the $H_0$-twisted Dorfman bracket into the $(H_0-db)$-twisted Dorfman bracket (cf. \cite{GRFbook} Proposition 2.14):
\begin{gather} \label{f:twistedDorfman}
[e^b(X+\xi),e^b(Y+\eta)]_{H_0-db}=e^b[X+\xi,Y+\eta]_{H_0}.
\end{gather}
We shall refer to $e^b$ as a \emph{$b$-field transformation} of $E$.

\begin{defn}[{\cite[Prop.\,2.21]{GRFbook}}]
The group of automorphisms of $(TM\oplus T^*M, [\cdot,\cdot]_{H_0}$) is
\[
\mathrm{Aut}(M,TM\oplus T^*M, [\cdot,\cdot]_{H_0})=
\{
\Phi\circ e^b\ |\ \Phi\in\mathrm{Diff}(M),\ b\in\Wedge^2(M),\ \Phi^*H_0=H_0-db
\}
\]
For simplicity, we will often denote this group $\mathrm{Aut}(M,H_0)$.  
\end{defn}

The group $\mathrm{Aut}(M,H_0)$ fits into an exact sequence
\[
1\to C^\infty(M,\Wedge_{\mathrm{cl}}^2(M))\to \mathrm{Aut}(M,H_0)\to \mathrm{Diff}(M)\to 1,
\]
where $\Wedge_{\mathrm{cl}}^2(M)$ is the sheaf of closed 2-forms. Its Lie algebra is
\[
\mathfrak{aut}(M,H_0)=\{X+b\ |\ d(i_X H_0+b)=0\},\qquad X\in C^{\infty}(M,TM), \, b\in C^{\infty}(M, \Wedge^2(M)),
\]
with the Lie bracket 
$[\cdot, \cdot]_{\rm aut}$ introduced above.
This Lie algebra fits into an exact sequence
\begin{gather}\label{eq:aut_sequence}
    \begin{split}
  0\to C^\infty(M,\Wedge_{\mathrm{cl}}^1(M))& \to C^\infty(E)\xrightarrow{\Psi} \mf{aut}(M,H_0)\to 0\\
  \Psi(X+\xi)=&\ X+d\xi-i_X H_0.
  \end{split}
\end{gather}
and
\[L_{\Psi(X + \xi)} (Y+ \eta) = [X+\xi, Y+\eta]_{H_0},\]
(which follows readily from \eqref{eq:Lie-derivative} and the definition of $\Psi$ above).
\begin{defn}\label{d:H-Lie-derivative} For sections $X+\xi$  and $Y + \eta$ of $E$, we let
\[ L^{H_0}_{X+ \xi} (Y+ \eta) := [X+ \xi, Y+\eta]_{H_0}.\]
Thus, $ L_{\Psi(X+ \xi)} (Y+\eta)= L^{H_0}_{X+ \xi} (Y+ \eta)$.
\end{defn}

\begin{rmk}
The notation $e^b$ will also be used for the formal exponential of $b$ as a differential form
\begin{align*}
    e^b = 1 + b + \tfrac{1}{2} b \wedge b + \dots,
\end{align*}
which in turn acts on $\Wedge^*(M)$ via the wedge product:
\[ \alpha \to e^b\wedge \alpha, \qquad \alpha \in \Wedge^*(M).\]
This gives rise to an action of $2$-forms on the spin bundle $\Wedge^*(M)$ compatible with the Clifford multiplication and $d_{H_0}$ in the following sense:
\begin{equation}\label{eq:b-field}
e^b(X+\xi)\cdot (e^{-b}\wedge\alpha)=e^{-b}\wedge((X+\xi)\cdot\alpha), \qquad  d_{H_0}\alpha=e^{b}\wedge d_{H_0-db}(e^{-b}\wedge\alpha).\end{equation}
Consequently, the action of $\mathrm{Aut}(M,H_0)$ naturally extends to the action on the spin bundle $\Wedge^*(M)$, and this action commutes with $d_{H_0}$.
\end{rmk}

\subsection{Divergence of generalized complex structures} \label{ss:GCdiv}

In this subsection we define a notion of divergence associated to a volume form $\vol$ and a generalized complex structure $\JJ$.  We first review basic points on generalized complex structures.

\begin{defn}
Let $(M,H_0)$ be a smooth manifold with a closed 3-form.  A \emph{generalized complex} structure on $(M, H_0)$ is a $\IP{,}$-orthogonal almost complex structure on $E$ whose $\i$-eigenbundle is integrable with respect to the $H_0$-twisted Dorfman bracket.  Let
$K_{\JJ}$ denote the canonical line bundle of $\JJ$, i.e. the complex line bundle locally spanned by pure spinors $\psi\in\Wedge^*(M)\otimes\C$ satisfying
\[
\Ker(\psi)=L^{1,0}_{\JJ},
\]
where $L^{1,0}_{\JJ} \subset E\otimes \C$ is the $\sqrt{-1}$-eigenspace of $\JJ$.  The integrability of $\JJ$ is equivalent to the condition that for every local defining spinor $\psi$ one has
\begin{align*}
    d_{H_0} \psi = e \cdot \psi
\end{align*}
for some $e \in E \otimes \mathbb C$.  Note that for a choice of defining spinor $\psi$, there exists a unique element $e \in E$, referred to as the \emph{real potential}, such that
\begin{align} \label{f:realpotential}
   d_{H_0} \psi = \frac{\i}{2} e \cdot \psi.
\end{align}
Furthermore, given $\vol$ a choice of volume form on $M$, we say that a defining spinor is \emph{normalized} if
\begin{equation}\label{eq:spinor_normalization}
(\psi,\bar{\psi})=(2\i)^n\vol. \end{equation}
\end{defn}

\begin{lemma} \label{l:divglemma} Given $(M, H_0)$, let $\JJ$ be a generalized complex structure and $\vol$ a volume form on $M$.  Choose $\psi$ a normalized local defining spinor and let $e = X + \xi$ be the real potential of $\psi$.  Then $X$ and $\Psi(e)=X+ d\xi - \imath_X H_0\in \mf{aut}(M,H_0)$ are globally defined.  Moreover, with respect to the generalized complex structure $\tilde \JJ = e^b \JJ e^{-b}$ on $(M, H_0-db)$ and volume form $\vol$, the corresponding real potential $\til{e}$ satisfies $\Psi(\tilde e)= X + d\xi -\imath_XH_0 + L_X b$.
\end{lemma}

\begin{proof} The normalization condition \eqref{eq:spinor_normalization} determines $\psi$ up to a multiplication with $e^{\i u}$, where $u$ is a local real function. If $\psi_u=e^{\i u}\psi$ is another normalized local section of $K_{\JJ}$, then
    \[
        d_{H_0} \psi_u = \frac{\i}{2} (e+ 2du)\cdot \psi_u.
    \]
    Clearly $\Psi(e + 2 du) = \Psi(e)$ and $\pi_{T} (e + 2 du) = \pi_T e$, as claimed.

    For the second claim, first note that $L_{\til{\JJ}}^{1,0} = e^b L_{\JJ}^{1,0}$, and then the integrability with respect to $H_0 - db$ follows from (\ref{f:twistedDorfman}).  Also, if $\psi$ is a defining spinor $\JJ$ then $e^{-b} \wedge \psi$ is a defining spinor for $\til{\JJ}$ by (\ref{eq:b-field}).  Furthermore note that it follows from the definition of the Mukai pairing that it is invariant under the $b$-field action, thus $e^{-b} \wedge \psi$ is $\vol$-normalized if $\psi$ is.  We also note from (\ref{eq:b-field}) that
    \begin{align*}
        d_{H_0 - db} (e^{-b} \wedge \psi) =&\ e^{-b} \wedge d_{H_0} \psi = e^{-b} \wedge (\i e \cdot \psi)\\
        =&\ \left( \i e^b (X + \xi) \right) \cdot e^{-b} \wedge \psi = \left( \i(X + i_X b + \xi) \right) \cdot e^{-b} \wedge \psi.
    \end{align*}
    Thus $\tilde{e} = X + i_X b + \xi$, and using the map $\Psi$ from (\ref{eq:aut_sequence}) associated to $H_0 - db$ we obtain
    \begin{align*}
        \Psi(\tilde{e}) = X + d (i_X b + \xi) - i_X (H_0 - db) = X + d \xi - i_{X} H_0 + L_X b.
    \end{align*}
\end{proof}

\begin{defn} Let $\JJ\in\End(E)$ be a generalized complex structure on $(M,H_0)$.  For a fixed volume form $\mathrm{vol}\in\Wedge^{2n}(M)$ we define the \emph{divergence of $\JJ$ relative to $\mathrm{vol}$}, denoted \[
\div_\vol(\JJ)\in \mf{aut}(M,H_0),
\]
as the unique element of $\mf{aut}(M, H_0)$ guaranteed by Lemma \ref{l:divglemma}.
\end{defn}

\begin{rmk} The notion of a `divergence operator' (cf. \cite{GF19}) plays an important role in the theory of connections in generalized geometry.  While apparently related, this notion is distinct from the concept defined above, though there should be some relation.  Note that the divergences used here are key to the definition of Goto's scalar curvature (cf. \S \ref{s:scalasmoment} below), which in turn by Theorem \ref{thm:scalformulaintro} is equal to the more general notion of scalar curvature in terms of a generalized metric and divergence operator \cite{GF19}.
\end{rmk}

\begin{rmk}
    Whenever a generalized complex structure $\JJ$ is determined by a \emph{closed} pure spinor $\psi$, we can define the corresponding spinorial volume form
    \[
    dV_{\psi}:=(2\i)^{-n}(\psi,\bar{\psi}).
    \]
    Clearly, in this case the divergence of $\JJ$ vanishes: $\div_{dV_{\psi}}(\JJ)=0$.
\end{rmk}

\begin{rmk}
    One can identify $\div_\vol(\JJ)$ with the unique $\JJ$-holomorphic unitary generalized connection $D\colon C^\infty(L)\to C^\infty(E\otimes L)$ on the line bundle $L=K_\JJ$ equipped with the Hermitian structure:
    \[
    |s|_{\vol}^2:=\frac{(2\i)^{-n}(s,\bar s)}{\vol}\in C^\infty(M,\R).
    \]
\end{rmk}

\begin{lemma} \label{l:divgpresJ}
    For any volume form $\vol$ the element $\div_{\vol} (\JJ) \in \mf{aut}(M,H_0)$ generates an infinitesimal symmetry of $\JJ$.
\end{lemma}
\begin{proof}
    Let $\psi$ be a local spinor defining $\JJ$ as in the definition of the divergence operator, and choose the unique $e \in E$ such that $d_{H_0}\psi=\i e\cdot \psi$. It suffices to prove that the well-defined infinitesimal action of $e$ on $\Wedge^*(M)\otimes\C$ preserves the subbundle $K_{\JJ}$. Indeed, computing its action on $\psi$ we find
    \[
    \begin{split}
    L^{H_0}_{e}\psi &= 
        (e\cdot d_{H_0}\psi+d_{H_0}(e\cdot \psi)) =\i e\cdot e\cdot\psi -\i d_{H_0}^2\psi=\i \la e,e\ra\psi.
    \end{split}
    \]
    Since $L_e^{H_0}$ preserves the $\C$-span of $\psi$, the underlying generalized complex structure is also preserved.
\end{proof}

A key point in the sequel is the relationship of this divergence to the canonical Poisson structure associated to a generalized complex structure.  

\begin{defn}\label{d:poisson} Given $\JJ$ a generalized complex structure, define the bivector $\pi_{\JJ} \in \Wedge^2 TM$ by
\begin{align*}
    \pi_{\JJ}(\ga,\gb) := \IP{ \JJ \ga, \gb}.
\end{align*}
It is shown in \cite{GualtieriThesis} that $\pi_{\JJ}$ is a real Poisson tensor.  Furthermore, $\poiss_{\JJ}$ is invariant under the conjugation $e^{b} \JJ e^{-b}$ of $\JJ$ by a $2$-form $b\in \Wedge^2(M)$.
\end{defn}

First observe that for a Poisson tensor and a fixed volume form there is a natural notion of divergence:

\begin{defn}\label{divergence-poisson}
    Given $(M, \pi)$ a Poisson manifold and a volume form $\vol$ on $M$, we define the \textit{divergence}
    \[
    X = \div_\vol(\pi)
    \]
    to be the unique vector field $X$
    such that
    \[
    d(i_{\pi} \vol)= i_{X} \vol.
    \]
    It follows directly from Cartan's formula, and its extension for the Schouten bracket, that
    \begin{equation}\label{eq:Cartan}
        L_X\pi=0,\quad L_X\vol=0.
    \end{equation}
\end{defn}

It turns out that the two notions of divergence are compatible in the following sense:

\begin{lemma} \label{l:divagree} Let $\JJ$ be a generalized complex structure on $(M, H_0)$.  Given a volume form $\vol$, one has
\begin{align*}
    \pi_T \div_{\vol}(\JJ) = \div_{\vol} (\pi_{\JJ}).
\end{align*}
\begin{proof} Observe by Lemma \ref{l:divglemma} that $\pi_T \div_{\vol}(\JJ)$ is invariant under the $b$-field action, and also $\pi_{\JJ}$ is invariant by definition.  Furthermore, as the regular points of a generalized complex structure are dense in $M$, in suffices to prove the result at regular points.  Using the Darboux theorem for generalized complex structures (\cite{GualtieriThesis} Theorem 4.35), in an open neighborhood of any regular point there exists a diffeomorphism and $b$-field transformation such that in this neighborhood $\JJ$ is a product of a standard complex space $\mathbb C^k$ and standard symplectic space $(\mathbb R^{2n - 2k}, \gw)$.  In particular it is defined by the closed pure spinor
\begin{align*}
    \psi = e^{\i \gw} \wedge \Theta,
\end{align*}
where $\Theta$ is a holomorphic volume form on $\mathbb C^k$.  It follows that $\pi_{\JJ} = \gw^{-1}$, where the inverse is defined leaf-wise.  It then follows easily, expressing $\vol = e^{-f} \Theta \wedge \bar{\Theta} \wedge \omega^k$, that the relevant normalized spinor is $e^{-f/2} e^{\i \gw} \wedge \Theta$, and thus $\pi_T \div_{\vol} (\JJ) = - \tfrac{1}{2} \omega^{-1} df = \tfrac{1}{2} \div_{\vol} (\pi_{\JJ})$.
\end{proof}
\end{lemma}

\begin{lemma}\label{l:invariant} Suppose $\JJ$ is a generalized complex structure on $(M, H_0)$ and $\vol$ a volume form. Suppose, furthermore, that $\pi_{\JJ}$ is non-degenerate over an open dense subset of $M$. Then $\div_{\vol}(\JJ) = \div_{\vol} (\pi_{\JJ})$ if and only if  $L_{\div_{\vol} (\pi_{\JJ})} \JJ =0$.
\end{lemma}
\begin{proof}
Let $X=\div_{\vol} (\pi_{\JJ})$ and $\div_{\vol} \JJ = X +B$, where $B$ is a $2$-form, see Lemma~\ref{l:divagree}.
In one direction, if $B$=0 then $L_X \JJ=0$ by Lemma~\ref{l:divgpresJ}.
In the other direction, Lemma~\ref{l:divgpresJ} and the assumption $L_{X}\JJ=0$, yield
\[0=L_{X + B}\JJ = L_B \JJ. \]
As the infinitesimal action of $b$-fields on generalized complex structures has no stabilizer in the locus where $\pi_{\JJ}$ is nondegenerate, we know $B=0$ on that locus, and hence everywhere by the assumption that this tensor is nondegenerate on a dense set.
\end{proof}

\begin{rmk} \label{r:divgremark1} The nondegeneracy condition on $\pi_{\JJ}$ in Lemma~\ref{l:invariant} is essential. For instance, consider the case when $\JJ$ is a generalized complex structure corresponding to a complex structure $J$ on $TM$, i.e 
\begin{align*}
    \JJ = \left( \begin{matrix}
        -J & 0\\
        0 & J^*
    \end{matrix} \right).
\end{align*}
In this case $\pi_{\JJ}=0$ and, therefore, $\div_{\vol} \pi_{\JJ}=0$ for any volume form $\vol$. On the other hand, we have a local spinor $\Theta$ for $\JJ$ which is a holomorphic volume form.  Now fix any volume form $\vol$ and note that this defines a function $f$ such that the spinor $\psi = e^{-f/2} \Theta$ is $\vol$-normalized.  An elementary computation using that $df - \i \JJ df \in \Ker (\psi)$ yields 
\begin{align*}
    d \psi = \i (- \JJ \tfrac{1}{2} df) \cdot \psi = \i (- \tfrac{1}{2} J^*  df) \cdot \psi.
\end{align*}
In particular, for generic $f$, $0 = \div_{\vol} (\pi_{\JJ}) \neq \div_{\vol} (\JJ) = - \tfrac{1}{2} d J^* df$.  Note though, using that $d J^* d f$ is of type $(1,1)$, it follows still that $L_{\div_{\vol} \JJ} \JJ=0$ in accordance with Lemma \ref{l:divgpresJ}.  Note furthermore that if $J$ is compatible with a K\"ahler metric $\gw$, and the data is interpreted as a generalized K\"ahler structure $(\JJ_{\gw}, \JJ)$, then for $f$ nonconstant $L_{\div_{\vol} \JJ} \JJ_{\gw} \neq 0$.
\end{rmk}

\subsection{Generalized K\"ahler structures}

\begin{defn}
Let $(M,H_0)$ be a smooth manifold with a closed 3-form.  A \textit{generalized K\"ahler} structure on $(M,H_0)$ is a pair of commuting generalized complex structures $(\II,\JJ)$ such that the bilinear form
\[
\la{-\II\JJ\cdot,\cdot \ra}
\]
on $TM\oplus T^*M$ is positive definite.  We denote the operator $-\II\JJ$ by $\G$, which is a \emph{generalized metric}.
\end{defn}

A fundamental result of Gualtieri states that a generalized K\"ahler structure $(\II,\JJ)$ on $(M,H_0)$ can be equivalently described by  bihermitian data on $(M,H_0)$.

\begin{thm}[\cite{GualtieriThesis} Ch. 6]\label{thm:gualtieri_map}
        Let $(\II,\JJ)$ be a generalized K\"ahler structure on $(M,H_0)$. Then there exist unique Riemannian metric $g$, real 2-form $b$ and orthogonal integrable complex structures $I$ and $J$,  such  that
    \begin{equation}\label{eq:gualtieri_map}
    \II = \frac{1}{2}e^{b}\left(
        \begin{matrix}I+ J & -(\omega_I^{-1} - \omega_J^{-1})\\ \omega_I -\omega_J & - (I^*+ J^*)\end{matrix}
    \right)e^{-b}, \quad \JJ =\frac{1}{2}e^{b}\left(
        \begin{matrix}I - J & -(\omega_I^{-1}+ \omega_J^{-1})\\ \omega_I + \omega_J & - (I^*- J^*)\end{matrix}
    \right)e^{-b},
    \end{equation}
    where the fundamental $2$-forms $\omega_I=gI$ and $\omega_J=gJ$ satisfy
    \begin{equation}\label{eq:gk_torsion}
    - d^c_I\omega_I=H_0+db= d^c_J\omega_J.
    \end{equation}
    Conversely, if the data $(g,b,I,J)$ satisfies~\eqref{eq:gk_torsion}, then~\eqref{eq:gualtieri_map} determines a generalized K\"ahler structure on $(M, H_0)$.
\end{thm}

\begin{rmk} By the theorem above we can refer equivalently to generalized K\"ahler structures by the data $(\II, \JJ)$ or the bihermitian data $(g, I, J, b)$.  At times we will refer simply to a generalized K\"ahler structure via $(g, I, J)$, where implicitly we mean to choose $H_0 = - d^c_I \gw_I = d^c_J \gw_J$ and $b = 0$.
\end{rmk}

\begin{defn} \label{d:nondeg} A generalized complex structure is called \emph{symplectic type} if the associated Poisson tensor $\pi_{\JJ}$ is everywhere nondegenerate.  A generalized K\"ahler structure $(\II, \JJ)$ on $(M, H_0)$ is called \emph{symplectic type} if $\JJ$ is symplectic type.  It is called \emph{nondegenerate} if $\II$ and $\JJ$ are both symplectic type.
\end{defn}

If the data $(g, I, J)$ satisfies that $\det(I+J)\neq 0$ on $M$, we can set $b:=g(I+J)^{-1}(I-J)$ and $F:= -2g(I+J)^{-1}$.  The data $(g, I, J, b)$ is then generalized K\"ahler with $H_0 = 0$, and furthermore $\JJ =\left(
        \begin{matrix}0 & -F^{-1}\\ F & 0\end{matrix}\right)$.  This implies that showing that $F$ is a symplectic form and $\JJ$ is a symplectic type generalized complex structure.

For a nondegenerate structure $(\II,\JJ)$ the corresponding transformations $I \pm J\in\End(TM)$ are invertible, and there are canonically associated symplectic forms
\begin{align*}
    F_{\pm} = -2 g (I \pm J)^{-1}.
\end{align*}
If $\II$ and $\JJ$ are given by~\eqref{eq:gualtieri_map}, then setting $b'=b-g(I+J)^{-1}(I-J)$ we observe that
\[
\II=e^{b'-4\Omega}
\left(\begin{matrix}0 & -F_-^{-1} \\ F_- & 0\end{matrix}\right)
e^{-b'+4\Omega},\quad 
\JJ=e^{b'}\left(\begin{matrix}0 & -F_+^{-1} \\ F_+ & 0\end{matrix}\right)e^{-b'}
\]
where $\Omega=g(I+J)^{-1}(I-J)^{-1}$ is a closed 2-form, see~\cite[\S 3.1]{ASnondeg}. The underlying $d_{H_0}$-closed spinors are $e^{4\Omega- b'+\i F_{-}}$ and $e^{- b'+\i F_+}$.

\section{Scalar curvature as moment map} \label{s:scalasmoment}
\subsection{Definition and basic properties}

In this subsection we recall the definition of Goto's scalar curvature following~\cite{goto-21}. Let $(\II,\JJ)$ be a generalized K\"ahler structure on $(M,H_0)$.  This data determines a Riemannian metric $g$ on $M$.  Let $dV_g$ be the corresponding volume form, and pick an arbitrary volume form $\mu_f:=e^{-f}dV_g$.  Associated to the data $(\II,\JJ,\mu_f)$, Goto assigns a scalar quantity $S(\II,\JJ,\mu_f)\in C^\infty(M,\R)$ generalizing the notion of the scalar curvature in K\"ahler geometry, determined purely by the underlying spinors and the measure $\mu_f$.  

\begin{defn} (\cite{goto-21}) \label{def:goto_scal}
Let $(\II,\JJ)$ be a generalized K\"ahler manifold on $(M, H_0)$, with a fixed volume form $\mu_f=e^{-f}dV_g$. Let $\psi_1$ and $\psi_2$ be pure spinors defining $\II$ and $\JJ$ respectively, normalized relative to $\mu_f$ according to~\eqref{eq:spinor_normalization}, with real potentials $\eta_i\in TM\oplus T^*M$ satisfying (\ref{f:realpotential}).  Then \emph{Goto's scalar curvature} $\Gscal(\II,\JJ,\mu_f)$ is the real function defined by the identity \begin{equation}\label{eq:goto_scalar_def}
\begin{split}
\Gscal (\II,\JJ,\mu_f)\mu_f
&:=
    \Real\left(
    (2\i)^{-n}(\psi_1,d_{H_0}(\eta_2\cdot \bar\psi_1))
    +
    (2\i)^{-n}(\psi_2,d_{H_0}(\eta_1\cdot\bar\psi_2))
    \right).
\end{split}
\end{equation}
\end{defn}

\begin{prop}\label{p:Gscal} ({\cite{goto-21} Propositions 3.2, 3.5, 3.6}) The following hold:
\begin{enumerate}
    \item $\Gscal(\II,\JJ,\mu_f)$ does not depend on the choice of the local spinors $\psi_i$ and determines a well-defined real-valued function on $M$.
    \item If $b$ is a $2$-form, then
    $\Gscal(\II,\JJ,\mu_f)=\Gscal(e^{b}\II e^{-b},e^{b}\JJ e^{-b},\mu_f)$.
    \item $\Gscal(\II,\JJ,\mu_f)=\Gscal(\JJ,\II,\mu_f)$.
\end{enumerate}
\end{prop}

\begin{rmk} Item (2) essentially follows from \cite{goto-21} Proposition 3.5, although as stated there it requires that $b$ is closed.  The general claim follows easily from the $b$-field transformation laws from \S \ref{s:background}, noting that the conjugated structure is generalized K\"ahler on the background $(M, H_0 - db)$.
\end{rmk}

\subsection{A formal moment map setup}

In \cite{Goto_2020, goto-21} (see also the preceding work by Boulanger~\cite{Boulanger}),  Goto interprets $\Gscal(\II, \JJ, \mu_f)$ as the moment map for a certain formal infinitesimal group action.  In particular, he defines a Lie algebra of Hamiltonian automorphisms using scalar potentials which are invariant under the divergence of $\pi_{\JJ}$ with respect to an \emph{arbitrary} volume form  $\mu_f$ (cf. Definition \ref{d:genHam}).  Then, a general variational formula for the scalar curvature is computed, and it is claimed that this determines the scalar curvature as the moment map for a class of `modified Hamiltonians' (cf. \cite{goto-21} Definition 4.2).  We were not able to verify that these modified Hamiltonians form a Lie algebra, and instead take a different approach here.  Inspired by recent works of E. Inoue~\cite{Inoue} of A. Lahdili \cite{Lahdili}, we ask that the volume  $\vol= \mu_f$ is such that the divergence of $\pi_{\JJ}$ with respect to $\vol$, which always preserves $\JJ$, also preserves $\II$.  This is a nontrivial further constraint (see Lemma~\ref{l:adapted-volume-BH} below) which  in the K\"ahler case $(M, g, J, J)$ boils down for asking that $\log\left(\frac{\vol}{dV_g}\right)$ is a Killing potential, which is precisely the set-up in \cite{Inoue}. In general, it is not always clear that such a volume form will exist, but with this further constraint in place Goto's variational formulas lead directly to the claim that the scalar curvature is the moment map for this infinitesimal group action.

\begin{defn} \label{d:GKadapted} Given a generalized K\"ahler structure $(\II, \JJ)$, we say that a volume form $\vol$ is \emph{adapted to $(\II, \JJ)$} if
\begin{align*}
    L_{\div_{\vol} \JJ} \II = 0.
\end{align*}
 For simplicity at times we will refer to a triple $(\II, \JJ, \vol)$ as \emph{adapted}.
 \end{defn}

\begin{lemma} \label{l:adaptedvolume} Suppose $(\II, \JJ)$ is a generalized K\"ahler structure defined on $(M, H_0)$, with corresponding bihermitian data $(g, I, J, b)$.  Then
\begin{align*}
    \pi_T \mf{aut}(M, H_0, \II, \JJ) < \mf{aut}(M, g, I, J).
\end{align*}
\begin{proof} Elements of $\mf{aut}(M, H_0)$ are canonically identified with sections $X+B$ of $TM \oplus \Wedge^2(T^*M)$.  The induced infinitesimal action of a section $X + B$ on the upper right block of an endomorphism of $E=TM \oplus T^*M$  is purely via Lie derivative by the vector field component $X$.  Thus if such an infinitesimal automorphism preserves $\II$ and $\JJ$, then it also preserves the generalized metric $\mathbb G = - \II \JJ$, and then by inspecting the form of $\II$ and $\JJ$ in Theorem \ref{thm:gualtieri_map}, it follows that $X$ preserves $g, I$, and $J$, as claimed. \end{proof}
\end{lemma}

 \begin{lemma}\label{l:adapted-volume-BH} Suppose $(\II, \JJ, \vol)$ is an adapted generalized K\"ahler structure on $(M, H_0)$.  Let  $(g, I, J, b)$ be the corresponding bihermitian data and $f$ the smooth function determined by  $\vol = \mu_f$.  Then the vector field
 \[ X=\div_{\vol} \pi_{\JJ}=\frac{1}{2}\left(I\theta_I^{\sharp} + J\theta_J^{\sharp} -(I+J) df^{\sharp}\right)\] 
 satisfies 
 \[ L_Xg=0, \qquad L_XI=L_XJ=0.\]
In particular, $X$ generates a compact torus $\T\subset\mathrm{Isom}(g)$.
 \end{lemma}
 
\begin{proof} Since $\div_{\vol}(\JJ)$ preserves $\II$ by hypothesis, and preserves $\JJ$ by Lemma \ref{l:divgpresJ}, Lemma~\ref{l:adaptedvolume} yields that 
\[ X:= \pi_T(\div_{\vol}(\JJ))\]
preserves $(g, I, J)$. We thus only need to compute $X$ in terms of $(g, I, J, \vol)$. 
By Lemma~\ref{l:divagree},
\[ X = \div_{\vol} \pi_{\JJ}\] whereas
by Theorem~\ref{thm:gualtieri_map}, 
\[ \pi_{\JJ}= \frac{1}{2}(I+J)g^{-1}.\]  Using the definition of $f$ and the bihermitian property,
\[ \vol=e^{-f}dV_g=e^{-f}\omega_I^{[n]}=\omega_J^{[n]},\]  thus
\[ \imath_{\pi_{\JJ}} \vol =\frac{e^{-f}}{2}\left(\omega_I^{[n-1]} + \omega_J^{[n-1]}.\right)\]
It follows that
\[ d(\imath_{\pi_{\JJ}}\vol) = \frac{e^{-f}}{2}\left((\theta_I - df)\wedge \omega_I^{[n-1]} + (\theta_J - df)\wedge \omega_J^{[n-1]}\right),\]
where $\theta_I:= Id^*\omega_I$ and $\theta_J:=Jd^*\omega_J$ are the Lee forms of $(g, I)$ and $(g, J)$, respectively, and we have used the basic identity $d\omega_I^{[n-1]}= \theta_I\wedge \omega^{[n-1]}$ and similarly for $\omega_J$.
It thus follows that
\[ X=\div_{\vol} \pi_{\JJ}=\frac{1}{2}\left(I\theta_I^{\sharp} + J\theta_J^{\sharp}\right) -\pi_{\JJ} df,\] 
as claimed.
\end{proof}

\begin{rmk}
    The existence of a Killing vector field $X$ for $(g, I, J)$ of the above form is thus a necessary condition for $(g, I, J)$ to come from a $\vol$-adapted GK pair $(\II, \JJ)$. In the case when $\pi_{\JJ}$ is non-degenerate on an open dense subset, this condition is also sufficient, see Lemma~\ref{l:invariant}.
\end{rmk}

\begin{defn}\label{d:adapted-BH} Given a generalized K\"ahler structure $(g, I, J)$, we say that the volume form $\vol$ is \emph{$(g,I,J)$-adapted} if the vector field $X$ defined in Lemma~\ref{l:adapted-volume-BH} preserves $(g, I, J)$.
\end{defn}

\begin{rmk} Remark \ref{r:divgremark1} makes clear that a volume form being $(g,I,J)$-adapated is strictly weaker than being adapted to $(\II, \JJ)$.
\end{rmk}

In what follows, we denote by $\T$ the torus generated by the infinitesimal isometry $X$ of the bihermitian structure $(g, I, J)$ corresponding to $(\II, \JJ, \vol)$.  
\begin{defn}\label{d:genHam} [Normalized potentials \& Hamiltonian automorphisms]
    Given $(\II, \JJ)$ a generalized K\"ahler structure with adapted volume form $\vol$ and corresponding torus $\T$, we define the space of \emph{normalized potentials} by
    \[
    C_0^\infty(M,\R)^\T=\{
        u \in C^\infty(M,\R)\ |\ du(X)=0,\ \int_M u\vol=0
    \}.
    \]
    Furthermore, define the space of \emph{Hamiltonian automorphisms} by
    \[
    \mf{ham}^\T(\JJ)=\{\Psi(\JJ du)\ |\ u\in C^\infty_0(M,\R)^\T\},
    \]
\end{defn}

\begin{lemma} (cf. \cite{Goto_LM} Proposition 3.1) \label{l:hamiltonians} Given the setup above, $\mf{ham}^{\T}(\JJ)$ is a Lie subalgebra of $\mf{aut}(M, H_0)$, isomorphic to $(C^{\infty}_0(M, \R), \vol)$ with the Lie bracket given by the Poisson bracket  with respect to $\pi_{\JJ}$:
\begin{align*}
    [\JJ du,\JJ dv]_{H_0} =\JJ d\{u,v\}_{\pi_{\JJ}}.
\end{align*}
Furthermore, $\mf{ham}^\T(\JJ)$ preserves $\JJ$, $\vol$,  and $\div_{\vol}(\JJ)$, i.e.
\begin{align*}
    L^{H_0}_{\JJ du} \JJ = 0,  \qquad L_{\poiss_{\JJ}du} \vol =0, \qquad { [ \Psi(\JJ du), \div_{\vol}(\JJ)]_{\rm aut} =0.}
\end{align*}
for all $u \in C_0^{\infty}(M, \mathbb R)^\T$. 
\begin{proof} Notice first that the map 
\[ (C^{\infty}_0(M, \R), \vol) \ni u \to \Psi(\JJ du) \in \mf{aut}(M, H_0)\]
is injective. Indeed, if $\Psi(\JJ du)=0$,  then by Theorem~\ref{thm:gualtieri_map} we find $\pi_T(\Psi(\JJ du))=\pi_{\JJ}(du)= \frac{1}{2}(I+ J) du^{\sharp}=0$ and $d(I-J)du=0$. It then follows that $dd^c_I u = dd^c_Ju=0$, so $u=0$ by the maximum principle and the normalization.

Using integrability of $\JJ$ and the definition of Dorfman bracket we have
\begin{align*}
    [\JJ du, \JJ dv]_{H_0} =&\ [du, dv]_{H_0} + \JJ [ \JJ du, dv]_{H_0} + \JJ [du, \JJ dv]_{H_0}\\
    =&\ \JJ [\JJ du, dv]_{H_0}\\
    =&\ \JJ L_{ \pi_{\JJ} du} dv\\
    =&\ \JJ d \{u, v\}_{\pi_{\JJ}}.
\end{align*}
If $X$ preserves $u, v$, and $\pi_{\JJ}$, then it also preserves $\{u, v\}_{\pi_{\JJ}}$, showing $\mf{ham}^{\T}(\JJ)$ is a subalgebra. Using  the integrability of $\JJ$ and the fact that the automorphisms induced by closed $1$-forms are trivial, it follows that
\begin{align*}
    L^{H_0}_{\JJ du} \JJ =&\ \JJ L^{H_0}_{du} \JJ = 0,
\end{align*}
as claimed.
Finally, we have 
\begin{equation}\label{eq:ad-invariance}
L_{\pi_{\JJ} du} \vol = d i_{\pi_{\JJ} du} \vol = d (d u \wedge i_{\pi_{\JJ}} \vol) = - du \wedge (d i_{\pi_{\JJ}} \vol) = (L_X u) \vol=0. \end{equation}
{The last property follows from the first two. } \end{proof}
\end{lemma}

The lemma above shows that $\mf{ham}^\T(\J)$ acts infinitesimally on the space of generalized K\"ahler structures $(\II, \JJ)$ adapted to $(\JJ, \vol)$.  The use of this space represents the key conceptual difference between our approach and that of \cite{goto-21}, where the `modified Hamiltonians' are used.  For technical purposes it is more natural to drop the integrability condition on $\II$ and consider the action of $\mf{ham}^{\T}(\JJ)$ on a larger space:

\begin{defn}{Given a generalized complex structure $\JJ$ on $(M, H_0)$ and a volume form $\vol$, the space of $(\JJ, \vol)$-\emph{adapted almost generalized K\"ahler structures} is 
the space of all generalized almost complex structures $\II$ compatible with $\JJ$ and invariant under the action of $\div_\vol(\JJ)$:
\[
\mc{AGK}(\JJ, \vol)=\left\{
    \II\in\End(E)\ |\ 
    \II^2=-\Id,\ \langle \II \cdot, \II \cdot \rangle = \langle \cdot, \cdot \rangle, \,  -\II\JJ=-\JJ\II>0,\ L_{\div_{\vol} (\JJ)} \II = 0 \right\}.
\]
        There is a natural K\"ahler structure (\cite{goto-21} Proposition 5.2, cf. also \cite{Boulanger}) determined at $\II\in\mc{AGK}(\JJ, \vol)$ by the symplectic form and Riemannian metric
\begin{align*}
\pmb\Omega_{\II}(\dot \II_1,\dot \II_2)=&\
\frac{1}{4}\int_M \tr{(\II \, \dot\II_1\dot\II_2})\vol, \qquad
\pmb{g} (\dot{\II}_1, \dot{\II}_2) = \frac{1}{4} \int_M \tr \left( \dot{\II}_1 \dot{\II}_2 \right) \vol.
\end{align*}}
The corresponding formal integrable almost complex structure acts by
\begin{equation}\label{eq:formal_complex}
    \boldsymbol{J}_{\II} (\dot \II) := \II \dot \II. \end{equation}
\end{defn}
It follows from Lemma~\ref{l:hamiltonians} that $\mf{ham}^\T(\JJ)$ acts on $\mc{AGK}(\JJ, \vol)$ via infinitesimal symmetry. It follows immediately from the work of Goto that this action is Hamiltonian, with moment map given by the scalar curvature.

\begin{thm} \cite{goto-21} \label{t:gotomoment} (Theorem \ref{t:gotomomentintro})
The natural action of $\mf{ham}^\T(\JJ)$ on $\mc{AGK}(\JJ, \vol)$ is Hamiltonian with moment map given by the scalar curvature  $\Gscal(\II, \JJ, \vol)$ associated to $(\II, \JJ, \vol)$ via  \eqref{eq:goto_scalar_def}.  In particular,
\begin{equation*}
\pmb{\Omega} \left(L^{H_0}_{\JJ du}\II, \left. \dt \right|_{t=0}\II_t \right)=
\int_M \left(\dt\Gscal(\II_t,\JJ,\vol)\right)u\vol
\end{equation*}
for any one-parameter family $\II_t\in \mc{AGK}(\JJ, \vol)$.
\end{thm}
\begin{proof}
    In~\cite{goto-21} Goto proved a more general variation formula
    \begin{equation}\label{eq:moment_map_Goto}
    \pmb{\Omega} \left(L^{H_0}_{\JJ du}\II,\left. \dt \right|_{t=0}\II_t \right)=
    \int_M \left(\left. \dt \right|_{t=0} \Gscal(\II_t,\JJ,\vol)\right)u\vol+2\pmb{\Omega} \left(u L_{\div_\vol(\JJ)}\II, \left. \dt \right|_{t=0}\II_t \right),
    \end{equation}
    which holds  for any $\T$-invariant smooth function $u$ on $M$.  In our setting $\div_\vol(\JJ)$ preserves  also $\II$, thus the second term on the right hand side vanishes.
\end{proof}

{  
\begin{rmk} In the above statement, we have identified the Lie algebra $\mf{ham}^\T(\JJ)$ with 
\[(C^{\infty}_0(M, \R)^{\T}, \vol, \{ \cdot, \cdot\}_{\pi_{\JJ}})\] via Lemma~\ref{l:hamiltonians}. Notice that $\Gscal(\II, \JJ, \vol)$ is a $\T$-invariant function by Lemma~\ref{l:adaptedvolume} and Theorem~\ref{thm:scalformula}.  The $L_2$ global product on $M$ with respect to $\vol$ defines an $\ad$-invariant inner product on $\mf{ham}^\T(\JJ)$ (see \eqref{eq:ad-invariance}), which we use to see $\Gscal(\II, \JJ, \vol)$ as an element of $(\mf{ham}^\T(\JJ))^*$.
\end{rmk}}

\begin{defn} \label{d:extremal} Given $(\II, \JJ, \vol)$ an adapted generalized K\"ahler structure on $(M, H_0)$, we say that
\begin{enumerate}
    \item $(\II, \JJ, \vol)$ is \emph{constant scalar curvature} (cscGK) if $\Gscal(\II, \JJ, \vol) = \bar{\mu}$ is constant.
    \item $(\II, \JJ, \vol)$ is \emph{extremal} if $\chi := \Psi(\JJ d \Gscal(\II, \JJ, \vol))$ is a Hamiltonian isometry, i.e. $L_{\chi} \II = 0$.
\end{enumerate}
\end{defn}

\begin{cor} \label{cor:extremal} Given $(\II, \JJ, \vol)$ an adapted generalized K\"ahler structure on $(M, H_0)$,
\begin{enumerate}
    \item $(\II, \JJ, \vol)$ is cscGK if and only if $\II$ is a zero of the moment map $\pmb{\mu}$ of 
    $\mathcal {AGK}(\JJ, \vol)$.
    \item $(\II, \JJ, \vol)$ is extremal if and only if $\II$ is a critical point of the functional $\II \to \brs{\brs{ \pmb{\mu}(\II)}}^2_{\vol}$ on $\mathcal {AGK}(\JJ, \vol)$.
\end{enumerate}
\begin{proof} The first claim is immediate from Theorem \ref{t:gotomoment}.  For the second claim we can compute the  variation and use Theorem \ref{t:gotomoment} to obtain
\begin{align*}
    \left. \frac{d}{dt}\right|_{t=0} \brs{\brs{\pmb{\mu}(\II_t)}}^2_{\vol} =&\ \int_M \left( \left.\dt \right|_{t=0} \Gscal(\II_t, \JJ, \vol) \right) \Gscal(\II, \JJ, \vol) \vol = \pmb{\Omega} \left( L_{\chi} \II, \left.\dt \right|_{t=0} \II_t \right).
\end{align*}
Since the symplectic form $\pmb{\Omega}$ is nondegenerate, the second claim follows.
\end{proof}
\end{cor}

\begin{rmk}\label{r:gauge} By Theorem~\ref{thm:scalformula}, $\Gscal(\II, \JJ, \vol)$ and hence also the corresponding notion of cscGK,   depends only on the data $(g, I, J, \vol)$. In other words, one can arbitrarily choose the $2$-form $b$ and set $H_0:= -d^c_I \omega_I - db$ (see \eqref{eq:gk_torsion}) in the momentum map setup, subject to the condition $L_{\div_{\vol}(\JJ)} \II=0$.   In the symplectic type case considered in \cite{ASU2}, the data $(g, I, J, \vol)$  is induced by a symplectic $2$-form $F$ on $M$ taming the complex structure $J$, with $\vol := \frac{F^{n}}{n!}$ being the corresponding symplectic volume.  In this case, letting $(H_0, b)=(0, (FJ)^{\rm skew})$ allows one to recast the setup in \cite{ASU2} in terms of the one presented here and in \cite{Goto_2020}.
\end{rmk}
In view of the above remark and Lemma~\ref{l:adaptedvolume}, we also give the following:
\begin{defn}\label{d:extremal-BH} Given a generalized K\"ahler structure $(g, I, J)$ and a $(g,I,J)$-adapted volume form $\vol$, we say that 
\begin{enumerate}
\item $(g, I, J, \vol)$ is cscGK if $\Gscal$ is constant;
\item $(g, I, J, \vol)$ is extremal if the vector field $\frac{1}{2}(I+ J)(d\Gscal)^{\sharp}$ preserves $(g, I, J)$.
\end{enumerate}
\end{defn}

\subsection{The generically symplectic case}

\begin{defn} We say that a generalized complex structure $\JJ$ on $(M, H_0)$ is \emph{generically symplectic type} if the real Poisson tensor $\pi_{\JJ}$ is non-degenerate over an open dense subset of $M$. A generalized K\"ahler structure 
$(\II, \JJ)$ is called \emph{generically symplectic type} if $\JJ$ is generically symplectic type. A bihermitian data $(g, I, J)$ is called \emph{generically symplectic type}  if  $\frac{1}{2}(I+J)g^{-1}$ is non-degenerate over an open dense subset.
\end{defn}
We note that a generalized K\"ahler structure for which the Poisson tensor defined by \eqref{eq:Poisson_BH} is non-degenerate over an open dense subset of $M$ are generically symplectic type. In particular, the generalized K\"ahler structures of even type on the compact complex surfaces  with first Betti number $1$ constructed in \cite{lebrun1991anti, fujiki2010anti} and \cite{SU1} are generically symplectic type.  Generalized K\"ahler structures of generically symplectic type  will allow us to extend the theory we develeped in the symplectic type case in~\cite{ASU2}.  This extension hinges on the following  observation.

\begin{lemma}\label{l:adapted_converse} Let $(M, g, I, J, b)$ be a generalized K\"ahler structure with adapted volume form $\vol$ in the sense of Definition ~\ref{d:adapted-BH}, and with corresponding generalized complex structures $(\II, \JJ)$.  Let $\T$ the compact torus generated by $X=\div_{\vol} (\pi_{\JJ})$. Suppose that $\JJ$ is generically symplectic type, and that $b$ is $\T$-invariant.  Then $(\II, \JJ, \vol)$ is adapted in the sense of Definition \ref{d:GKadapted}.
\end{lemma}
\begin{proof} By Theorem~\ref{eq:gualtieri_map}, we have that $L_X \JJ=L_X \II =0$.
Thus, by Lemma~\ref{l:invariant}, 
$\div_{\vol}(\JJ)  = X$ 
and, therefore, 
\[\qquad L_{\div_{\vol}(\JJ)} \II = L_X \II=0.\]
\end{proof}

\begin{rmk} \label{r:divgremark} Using Lemma~\ref{l:adapted_converse}, we can start with a generically symplectic type bihermitian data $(g, I, J)$ on $M$,  and an adapted volume form $\vol$ generating a torus $\T$. We then choose a $\T$-invariant $2$-form $b$ (note $b = 0$ will suffice) and define $H_0:= -d^c_I \omega_I - db$. Then, by Lemma~\ref{l:adapted_converse}, the corresponding generalized K\"ahler structure $(\II, \JJ)$ on $(M, H_0)$ belongs to $\mathcal{AGK}(\JJ, \vol)$. Furthermore, 
by virtue of Lemma~\ref{l:invariant}, in this case we have  $\div_{\vol}(\JJ)=X$ and therefore 
\begin{equation*} \mathcal{AGK}(\JJ, \vol)=\mathcal{AGK}^{\T}(\JJ), \end{equation*}
where $\mathcal{AGK}^{\T}(\JJ)$ is the space of all $\JJ$-compatible $\T$-invariant almost generalized K\"ahler pairs $(\II, \JJ)$ for the natural action of $\T$ on $E=TM \oplus T^*M$ by diffeomorphisms. \end{rmk}

\section{Generalized K\"ahler-Ricci solitons as cscGK structures}

In this section we show that generalized K\"ahler-Ricci solitons are automatically Goto constant scalar curvature, and fit them into the GIT framework constructed above.  The first step is to obtain an explicit formula for this curvature in terms of the bihermitian data.  Given this, we derive a Bianchi identity for the weighted Bismut Ricci curvature which gives the claim that solitons have constant scalar curvature.  We next derive an a priori symmetry which holds for solitons which identifies the relevant adapted volume form, giving the GIT interpretation for generically symplectic type solitons.

\subsection{An explicit formula}

 In this subsection we derive an explicit formula for $\Gscal(\II,\JJ,\mu_f)$ in terms of the underlying bihermitian data $(M,g,I,J,b)$.  To this end,  we first establish the formula under the further key hypothesis that the structure is \emph{nondegenerate}, see Definition~\ref{d:nondeg}.  The reason for this hypothesis is that in this setting the pure spinors underlying the generalized K\"ahler structure can be made completely explicit in terms of the bihermitian data.  In the general case, these underlying pure spinors would involve the choice of partial holomorphic volume forms with respect to both $I$ and $J$, and it is unclear how to work directly with the mixed $I$-holomorphic and $J$-holomorphic coordinates.  Given the formula in this special case, the general case follows using the nondegenerate perturbation technique introduced in \cite{AFSU}.
 
We begin with a preliminary computational lemma.  In the statements below different Laplacians appear.  In particular we let $\gD$ denote the Riemannian Laplacian, while $\gD^{C,I}$ will denote the Laplacian of the Chern connection of the Hermitian structure $(g, I)$, and analogously $\gD^{C,J}$. Recall that for $f\in C^\infty(M,\R)$
\[
\gD^{C,I}f=\gD f-\la{df,\theta_I \ra},
\]
where $\theta_I=Id^*\omega_I$.

\begin{lemma}\label{lm:F_trace} Let $(M, g, I, J)$ be a generalized K\"ahler structure with $I+J$ invertible, and define $F_+=-2g(I+J)^{-1}$. Denote by $d_{F_+} := F_+g^{-1} d$ the twisted differential. Then, for any smooth function $f$, 
\[ \begin{split}
\frac{(d d_{F_+} f) \wedge F_+^{n-1}}{F_+^{n}} &= \Delta f -\frac{1}{2}\big\langle d f, d \log \det(I+J) \big\rangle.
\end{split}\]
Similarly, whenever $(I-J)$ is invertible, we define $F_- = -2 g(I - J)^{-1}$ and have
\[ \begin{split}
\frac{(d d_{F_-} f) \wedge F_-^{n-1}}{F_-^{n}} &= \Delta f -\frac{1}{2}\big\langle d f, d \log \det(I-J) \big\rangle.
\end{split}\]

\end{lemma}
\begin{proof} In the computation below,  the upper-script  $\sharp$ denotes $g^{-1}$, and we use that  $F_+=-2g(I+J)^{-1}$, $b=g(J+I)^{-1}(I-J)$, $\theta_I = I \delta^g I$, $\theta_J= J \delta^g  J$.
\[
\begin{split}
\frac{(d d_{F_+} f) \wedge F_+^{n-1}}{F_+^{n}} &=\sum_{i=1}^{2m}  \Big\langle(I+J) \left(\nabla_{e_i} (I+J)^{-1} d f ^{\sharp}\right), e_i \Big\rangle\\
      &=\Delta f - \sum_{i=1}^{2n}\Big\langle \left(\nabla_{e_i} (I+J)\right)\left((I+J)^{-1}(d f^{\sharp})\right), e_i\Big\rangle\\
      &= \Delta  f -\big\langle (I+J)^{-1}(d f^{\sharp}), \delta^g(I + J)\big\rangle\\
      &= \Delta  f -\big\langle d f^{\sharp}, (I+J)^{-1}(I\theta_I^{\sharp} + J\theta_J^{\sharp})\big\rangle\\
      &= \Delta  f + \frac{1}{2}\big\langle d f , b(\theta_J^{\sharp} - \theta_I^{\sharp})- \theta_I - \theta_J\big\rangle.
\end{split}\]
Our claim then follows from the identity 
\begin{equation*}
  \frac{1}{2}d\log \det (I+J)= \theta_J - \imath_{\theta_J^{\sharp}} b - \delta^g b =  \theta_I  +\imath_{\theta_I^{\sharp}} b  + \delta^g b,   \end{equation*}
which is established in \cite[Lemmas~2.10 \& 2.13]{ASnondeg}.
\end{proof}

\begin{lemma}\label{lemma:goto_scal_nondeg}
    Let $(M,g,I,J)$ be a nondegenerate generalized K\"ahler manifold.  Then for $f \in C^{\infty}(M)$ one has
    \begin{equation}\label{eq:goto_scal_nondeg_psi}
    \Gscal(\II,\JJ,\mu_f)=\frac{1}{4}
        \left(
            R - \frac{1}{12} \brs{H}^2 + 2\Delta f-|df|^2 \right),
\end{equation}
where $H=H_0 + db = - d^c_I \omega_I= d^c_J \omega_J$.
\end{lemma}

\begin{proof} Since $(M,g,I,J)$ is nondegenerate,  both $I+J$ and $I-J$ are globally invertible. We define $b = g(I + J)^{-1} (I - J)$, and apply this $b$-field transform to $(\II, \JJ)$. By Theorem \ref{thm:gualtieri_map} the corresponding generalized K\"ahler structure, still denoted $(\II,\JJ)$, is given by
\begin{equation}\label{eq:gk_nondeg_gauged}
\II=e^{-4\Omega}
\left(\begin{matrix}0 & -F_-^{-1} \\ F_- & 0\end{matrix}\right)
e^{4\Omega},\quad 
\JJ=\left(\begin{matrix}0 & -F_+^{-1} \\ F_+ & 0\end{matrix}\right),
\end{equation}
where $F_\pm = -2g(I\pm J)^{-1}$ are the underlying symplectic forms, and $\Omega=g(I+J)^{-1}(I-J)^{-1}$ is the common real part of $I$- and $J$-holomorphic symplectic forms.  Since $\Gscal(\II,\JJ,\mu_f)$ is invariant under $b$-field transformation, we can calculate it for the GK structure given by~\eqref{eq:gk_nondeg_gauged}.  It is easy to see that the $\i$-eigenspaces for $\II$ and $\JJ$ are
\[
L_{\II}^{1,0}=\{ v - 4\Omega(v) - \i F_-(v)\ |\ v\in TM \},\quad
L_{\JJ}^{1,0}=\{ v - \i F_+(v)\ |\ v\in TM \}.
\]
These generalized complex structures have underlying closed pure spinors
\[
\varphi_1=\exp(4\Omega+\i F_-),\quad \varphi_2=\exp(\i F_+)
\]
with Mukai norms
\[
\begin{split}
(\varphi_1,\bar{\varphi_1})=\frac{F_-^n}{n!}=\det(I-J)^{-1/2}dV_g=e^{-\Psi_-}dV_g,\\ (\varphi_2,\bar{\varphi_2})=\frac{F_+^n}{n!}=\det(I+J)^{-1/2}dV_g=e^{-\Psi_+}dV_g,
\end{split}
\]
where by definition
\begin{align*}
    \Psi_{\pm}:=\frac{1}{2}\log\det(I\pm J).
\end{align*}
In the definition of the scalar curvature, we need to use normalized spinors, thus we introduce
\begin{equation*}
\begin{split}
    f_1=-f/2+\log\det(I-J)/4, \qquad f_2=-f/2+\log\det(I+J)/4,
\end{split}
\end{equation*}
so that $\psi_i := e^{f_i}\varphi_i$ satisfies
\begin{align*}
(\psi_i,\bar{\psi_i})=\mu_f=e^{-f}dV_g.
\end{align*}
Since $\varphi_i$ are closed and $df_i-\i\J_idf_i\in \Ker\psi_i$, we have \[
\begin{split}
d\psi_i &= df_i\wedge \psi_i=df_i\wedge \psi_i-(df_i-\i\J_idf_i)\cdot \psi_i = \i(\J_idf_i)\cdot \psi_i.
\end{split}
\]
Using the explicit forms of $\J_i$, we find that the real potentials $\eta_i=\J_idf_i$ are
\[
\begin{split}
\eta_1 &= F_-^{-1}(df_1)-4\Omega F_-^{-1}(df_1)=-\frac{1}{2}(I-J)g^{-1}df_1+2g(I+J)^{-1}g^{-1}df_1,\\
\eta_2 &= F_+^{-1}(df_2)=-\frac{1}{2}(I+J)g^{-1}df_2.
\end{split}
\]
We also observe the basic fact that the Clifford involution $\gs$ in the definition of the Mukai pairing satisfies
\[
\gs \bar{\psi_1} = \gs (e^{f_1} \bar{e^{4 \Omega + \i F_-}}) = e^{f_1} \gs (e^{4 \Omega - \i F_-}) = e^{f_1} e^{- 4 \Omega + \i F_-}.
\]
With these formulas we now can compute the necessary terms from~\eqref{eq:goto_scalar_def}.  We start with the first one, the second being analogous.  Using Lemma \ref{lm:F_trace} we have
\begin{equation*}
    \begin{split}
        \Real (\psi_1,d({\i \eta_2}\cdot\bar\psi_1))&=
        \Real \left[
            (2\i)^{-n}e^{f_1}e^{4\Omega+\i F_-}\wedge d(e^{f_1}\imath_{\i \eta_2}(-4\Omega+\i F_-))\wedge e^{-4\Omega+\i F_-}
        \right]_{\mathrm{top}}\\
        &=
        -\frac{1}{2}e^{2f_1}F_-^{n-1}\wedge \left(df_1\wedge F_-(g^{-1}df_2) + dd_{F_-}f_2 \right)\\
        &=
        -\frac{1}{2}\frac{\left(df_1\wedge F_-(g^{-1}df_2) + dd_{F_-}f_2 \right)\wedge F_-^{n-1}}{F_-^{n}}\mu_f\\
        &=
        -\frac{1}{2}\left(
            \la{df_1,df_2 \ra}+\Delta f_2-\frac{1}{2}\la{df_2,d\log\det(I-J) \ra}
        \right)\mu_f\\
        &=
        \frac{1}{8}
\left(
    2\Delta (f - \Psi_+)-|df|^2+\la{d\Psi_+,d\Psi_- \ra}+\la {df, d\Psi_- -d\Psi_+ \ra}
\right)\mu_f,
    \end{split}
\end{equation*}
where the last line follows using the definitions of $\Psi_{\pm}$.
We can compute the second summand $\Real (\psi_2,d({\eta_1}\cdot\bar\psi_2))$ of Definition~\ref{def:goto_scal} analogously, or just observe that it is invariant under the action of $b$-field transform, and the conjugation by $e^{4\Omega}$ transforms it into an identical expression, modulo exchanging $\Psi_{\pm}$.  This yields
\[
\Real (\psi_2,d({\i \eta_1}\cdot\bar\psi_2))=\frac{1}{8}
\left(
    2\Delta (f - \Psi_-)-|df|^2+\la{d\Psi_+,d\Psi_- \ra}+\la {df, d\Psi_+ -d\Psi_- \ra}
\right)\mu_f.
\]
Finally, there is the further identity (\cite{AFSU} Lemma 4.5)
    \begin{align*}
        R - \tfrac{1}{12} \brs{H}^2 = - \Delta (\Psi_+ + \Psi_-) + \IP{d \Psi_+, d \Psi_-}.
    \end{align*}
Collecting the above calculations gives the result.
\end{proof}

\begin{thm}\label{thm:scalformula} (Theorem \ref{thm:scalformulaintro})
    Let $(M,g,I,J)$ be a generalized K\"ahler manifold determined by generalized complex structures $(\II,\JJ)$.  Then for all $f \in C^{\infty}(M)$ one has
    \begin{align*}
        \Gscal(\II,\JJ,\mu_f)=\frac{1}{4}\left(
            R-\frac{1}{12}|H|^2+2\Delta f-|df|^2
        \right).
    \end{align*}
    \begin{proof} This follows from the nondegenerate perturbation.  The details are identical to that used for instance in (\cite{AFSU} Proposition 4.3), so we give a brief sketch.  By taking a product with a flat $T^2$ if necessary it suffices to consider the case where both underlying generalized complex structures $\II$ and $\JJ$ have even type.  Furthermore, as the left and right hand side of the claimed identity are both global smooth functions, it suffices to establish the identity on the open dense set of points where $\II$ and $\JJ$ have locally constant type.  In this locus \cite{AFSU} Theorem 3.4 can be applied to construct a sequence of nondegenerate GK structures limiting to $(\II, \JJ)$ in $C^{k,\ga}$.  The claimed formula holds for the approximate structures by Lemma \ref{lemma:goto_scal_nondeg}.  The convergence of the GK structures to the pair $(\II, \JJ)$ induces smooth converges of the underlying bihermitian data, hence the claimed equation passed to the limit, finishing the proof.
    \end{proof}
\end{thm}

\subsection{Generalized Ricci solitons}

Given a smooth manifold $M$, a triple $(g, H, f)$ of a Riemannian metric, closed three-form, and function defines a \emph{generalized Ricci soliton} (cf. \cite{GRFbook}) if
\begin{align*}
\Rc - \tfrac{1}{4} H^2 + \N^2f =&\ 0\\
d^*_g H + i_{\N f} H =&\ 0.
\end{align*}
These equations are naturally expressed using the associated Bismut connection,
\begin{align*}
    \N^H := \N + \tfrac{1}{2} g^{-1} H.
\end{align*}
In particular, if we let $\Rc^{H}$ denote the Ricci curvature of this connection, we define
\begin{align*}
    \Rc^{H,f} :=&\ \Rc^{H} + \N^H \N^H f = \Rc - \tfrac{1}{4} H^2 + \N^2 f - \tfrac{1}{2} (d^*_g H + i_{\N f} H),\\
    R^{H,f} =&\ R - \tfrac{1}{12} \brs{H}^2 + 2 \gD f - \brs{\N f}^2.
\end{align*}
and we observe that the generalized Ricci soliton system is equivalent to the vanishing of $\Rc^{H,f} \equiv 0$.  The relevant identity uses the weighted divergence operator associated to $H$ and $f$, i.e. for a covariant tensor $T$ we define
	\begin{align*}
	\div^{H,f} T = e^{f} \tr \N^H \left( e^{-f} T \right).
	\end{align*}
With this background we can prove a Bianchi identity for these curvatures, which is a special case of (\cite{SCSV} Proposition 6.8):

\begin{prop} \label{p:Bianchi} Let $(M^n, g, H,f)$ be a Riemannian manifold with closed three-form and function $f$.  Then
	\begin{align*}
	\div^{H,f} \Rc^{H,f} =&\ \tfrac{1}{2} \N R^{H,f}.
	\end{align*}
	\begin{proof} We first observe that for a two-tensor $K$,
		\begin{align*}
		\N^H_i K_{ij} = \N_i K_{ij} + \tfrac{1}{2} H_{ipj} K_{ip}.
		\end{align*}
		Next observe using the classical Bianchi identity and the Bochner identity that
		\begin{align*}
		\N^{f}_i \Rc^{f}_{ij} =&\ \N_i \left( \Rc_{ij} + \N_i \N_j f \right) - \Rc^{f}_{ij} \N_i f\\
		=&\ \tfrac{1}{2} \N_j R + \N_i \N_j \N_i f - \Rc_{ij} \N_i f - \N_i \N_j f \N_i f\\
		=&\ \tfrac{1}{2} \N_j \left( R + 2 \gD f - \brs{\N f}^2 \right).
		\end{align*}
		Also using (\cite{GRFbook} Lemma 3.19) we obtain
		\begin{align*}
		\N_i \left( - \tfrac{1}{4} H^2 - \tfrac{1}{2} d^* H \right)_{ij} + \tfrac{1}{2} H_{ipj} \left(-\tfrac{1}{2} d^* H \right)_{ip} = - \tfrac{1}{24} \N_j \brs{H}^2.
		\end{align*}
		Lastly, observe
		\begin{align*}
		- \tfrac{1}{2} \N_i (i_{\N f} H)_{ij} =&\ - \tfrac{1}{2} \N_i \left( \N_k f H_{kij} \right)= - \tfrac{1}{2} \N_k f \N_i H_{kij} = - \tfrac{1}{2} \N_k f (d^* H)_{kj}.
		\end{align*}
		Combining these, using $d^* d^* H = 0$ and observing a further elementary cancellation we conclude
		\begin{align*}
		\N^{H,f}_i \Rc^{H,f}_{ij} =&\ \N^H_i \Rc^{H,f}_{ij} - \Rc^{H,f}_{ij} \N_i f\\
		=&\ \N_i \Rc^{H,f}_{ij} + \tfrac{1}{2} H_{ipj} \Rc^{H,f}_{ip} - \Rc^{H,f}_{ij} \N_i f\\
		=&\ \N_i \left(\Rc - \tfrac{1}{4} H^2 +\N^2 f - \tfrac{1}{2} \left( d^*_g H + i_{\N f} H \right) \right)_{ij} + \tfrac{1}{2} H_{ipj} \left( - \tfrac{1}{2} \left( d^* H + i_{\N f} H \right) \right)_{ip}\\
		&\ - \left( \Rc - \tfrac{1}{4} H^2 +\N^2 f - \tfrac{1}{2} \left( d^*_g H + i_{\N f} H \right) \right)_{ij} \N_i f\\
		=&\ \tfrac{1}{2} \N_j R^{H,f}.
		\end{align*}
	\end{proof}
\end{prop}

Proposition \ref{p:Bianchi} has the immediate consequence that generalized Ricci solitons have constant twisted scalar curvature $R^{H,f}$.  We note that this result is also implicit by the scalar curvature monotonicity result for generalized Ricci flow \cite{Streetsscalar}, and also follows from (\cite{GRFbook} Proposition 4.33).

\begin{cor} \label{c:GRScor} Given $(g, H, f)$ a generalized Ricci soliton, then
\begin{align*}
    R^{H,f} \equiv&\ \lambda,\\
    \tfrac{1}{6} \brs{H}^2 + \gD f - \brs{\N f}^2 \equiv&\ \lambda.
\end{align*}
Furthermore, if $M$ is compact, then $\lambda \geq 0$ with equality if and only if $H \equiv 0, \N f \equiv 0$.
\begin{proof} The claim that $R^{H,f}$ is constant is immediate from Proposition \ref{p:Bianchi}.  Calling this constant $\lambda$, we then obtain
\begin{align*}
    \lambda =&\ R^{H,f} - \tr_g \Rc^{H,f} = \tfrac{1}{6} \brs{H}^2 + \gD f - \brs{\N f}^2.
\end{align*}
If $M$ is compact, we integrate this final equation against $e^{-f} dV_g$ to obtain
\begin{align*}
    \lambda \int_M e^{-f} dV_g =&\ \tfrac{1}{6} \int_M \brs{H}^2 e^{-f} dV_g.
\end{align*}
It follows that $\lambda \geq 0$ with equality if and only if $H \equiv 0$.  In this case $(g, f)$ is a compact steady soliton, and then it is known that also $\N f \equiv 0$.
\end{proof}
\end{cor}

\subsection{Generalized K\"ahler-Ricci solitons}

\begin{defn}[Generalized K\"ahler-Ricci solitons] A generalized K\"ahler structure $(g, I, J)$ is called a \emph{generalized K\"ahler-Ricci soliton} if there exists a smooth function $f$ such that the data $(g, H, f)$ defines a generalized Ricci soliton. 
\end{defn}

Note that Corollary \ref{c:GRScor} and Theorem~\ref{thm:scalformula} yield that any generalized K\"ahler-Ricci soliton has constant Goto scalar curvature $\Gscal(\II, \JJ, \mu_f)$.  We are next show that the volume form $\vol = \mu_f$ is adapted, i.e. $(g, I, J, e^{-f}dV_g)$ is a cscGK structure in the sense of Definition~\ref{d:extremal-BH}.  It is shown in (\cite{SU2} Prop 4.1) that for any generalized K\"ahler-Ricci soliton,  the vector fields
\begin{align*}
X_I := \tfrac{1}{2} I \left(\theta_I^{\sharp} - \N f \right), \quad X_J := \tfrac{1}{2} J \left(\theta_J^{\sharp} - \N f \right)
\end{align*}
satisfy
\begin{align} \label{f:solitonvf}
    L_{X_I} g = L_{X_J} g = 0, \qquad L_{X_I} I = L_{X_J} J = 0.
\end{align}
We upgrade this to 

\begin{lemma} \label{l:solitonID} Let $(M^{2n}, g, I, J, f)$ be a generalized K\"ahler-Ricci soliton and let $X = X_I + X_J$.  Then
\begin{align*}
    L_{X} I = L_X J = 0, \, L_X g=0, \qquad [X_I, X_J] = 0.
\end{align*}
\begin{proof}  {It is shown in (\cite{ASU3}, Proposition 4.6) that one has the identities (recall $\pi$ here is given by \ref{eq:Poisson_BH})
\begin{align*}
    L_{I X_I} \poiss = L_{J X_J} \poiss = L_{X_I} \poiss = L_{X_J} \poiss = 0.
\end{align*}
 As $X= \frac{1}{2} (I\theta_I^{\sharp} + J\theta_J^{\sharp})- \pi_{\JJ}(df)= \div_{e^{-f}dV_g} \pi_{\JJ}$ (see Lemma~\ref{l:adapted-volume-BH}),  we have   $L_X \pi_{\JJ}=0$ (see \eqref{eq:Cartan}). Using this and (\ref{f:solitonvf}), we obtain
\begin{align*}
    L_{X_I} J + L_{X_J} I = 0.
\end{align*}
It then follows that
\begin{align*}
    0 =&\ L_{X_I} (2 \pi g) = [L_{X_J} I, J] = [J, L_{X_I} J] = 2 J L_{X_I} J.
\end{align*}
Thus $L_{X_I} J \equiv 0$ and similarly $L_{X_J} I \equiv 0$.  Note it follows then immediately that $L_{X_I} \theta_I = L_{X_I} \theta_J = L_{X_J} \theta_I = L_{X_J} \theta_J = 0$, and also $L_{X_I} H = L_{X_J} H = 0$.  Then we can compute
\begin{align*}
    L_{X_I} (- J X_J) = \tfrac{1}{2} [\theta_{J}^{\sharp} - \N f, X_I] = [X_I, I X_I] + \tfrac{1}{2} [\theta_j^{\sharp} - \theta_I^{\sharp}, X_I] = 0,
\end{align*}
hence $0 = L_{X_I} (J X_J) = J [X_I, X_J]$, finishing the proof.}
\end{proof}
\end{lemma}

\begin{cor} \label{c:GITforGKRS} Let $(M^{2n}, g, I, J,f)$ be a generalized K\"ahler-Ricci soliton.
Then the volume form $\vol = \mu_f$ is $(g,I,J)$-adapted and $(g, I, J, \mu_f)$ is a cscGK structure.
\begin{proof} This is an immediate corollary of Corollary~\ref{c:GRScor} and Lemmas~\ref{l:adapted-volume-BH} and \ref{l:solitonID}. \end{proof}
\end{cor}

\begin{cor} \label{c:GITgensymp} Let $(M^{2n}, g, I, J,f)$ be a generalized K\"ahler-Ricci soliton of generically symplectic type. Denote by $\T$ the torus generated by the adapted volume form $\vol=\mu_f$ and let $b$ be any $\T$-invariant $2$-form. Then the generalized K\"ahler structure $(\II, \JJ)$ on $(M, H_0=-d^c_I \omega_I-db)$ corresponding to $(g, I, J, b)$
is cscGK in $\mc{AGK}(\JJ, \vol)=\mc{AGK}^{\T}(\JJ)$.
\begin{proof}
The hypothesis that the structure is generically symplectic implies, by Lemma~\ref{l:invariant}, that the triple $(\II, \JJ, \vol)$ is adapted.
\end{proof}
\end{cor}

\begin{ex} \label{r:Liegroupscal} As an immediate consequence of Theorem~\ref{thm:scalformula}, we recover the result of Goto~\cite{goto-21} that the canonical generalized K\"ahler structures on semisimple Lie groups have constant scalar curvature.  In particular, as described in \cite{GualtieriThesis}, given a compact semisimple Lie group $K$ with bi-invariant metric $g$, $H$ the Cartan three-form and $J_L$ and $J_R$ left- and right-invariant complex structures compatible with $g$, then $(g, H, J_L, J_R)$ is generalized K\"ahler.  By the invariance it follows easily that $\Gscal(\II, \JJ, dV_g) = R - \tfrac{1}{12} \brs{H}^2 = \tfrac{1}{6} \brs{H}^2$ is a positive constant.  Furthermore, in case $K$ has no abelian factors, the generalized K\"ahler structure is generically symplectic type.  In this case $dV_g$ is adapted to $(\II, \JJ)$ by Corollary \ref{c:GITgensymp}, and thus the data is cscGK in $\mc{AGK}(\JJ, dV_g)$.
\end{ex}

\begin{ex} In a series of works \cite{Streetssolitons, SU1, ASU3} the authors constructed and classified all generalized K\"ahler-Ricci solitons on compact complex surfaces, with the non-K\"ahler examples occurring on Hopf surfaces.  Furthermore, in \cite{SU2} the authors gave an exhaustive construction complete generalized K\"ahler-Ricci solitons on surfaces with minimal symmetry group via an extension of the Gibbons-Hawking ansatz.  These metrics are thus all cscGK, while none of these examples are symplectic type (the case covered by our earlier work \cite{ASUScal}), they are all generically symplectic type.  Thus by Corollary \ref{c:GITgensymp} they admit adapted volume forms $\vol = \mu_f$ and are cscGK in $\mc{AGK}(\JJ, \vol)$.
\end{ex}

\section{Generalized K\"ahler class as complexified orbit in the generically symplectic case}\label{s:orbit}
\subsection{Generalized K\"ahler class}

Let $(\II,\JJ)$ be a GK structure corresponding to the bihermitian data $(g,I,J,b)$.  We are going to introduce a notion of a generalized K\"ahler class $\mc{GK}_{\pi, J}$ fixing $J$ and $\pi=\frac{1}{2}g^{-1}[I,J]$ and varying $I$, $g$ and $b$.  Our discussion builds on the fundamental works \cite{BGZ, GualtieriBranes}.

\begin{defn}\label{d:GK-class}[Generalized K\"ahler class]
    Consider a GK structure $m_0=(g_0,I_0,J,b_0)$.
    Define the \textit{generalized K\"ahler class based at $m_0$} as
    \begin{equation}\label{eq:def_gk_class}
        \GK_{m_0}(\pi,J)=\{
        \mbox{exact }A\in\Wedge^2(M)\ |\ AI_0+I_0^*A-A\pi A=0
        \}.
    \end{equation}
    As was observed by Gualtieri~\cite{GualtieriBranes} any such $A$ gives rise to a GK structure $(g_1,I_1,J,b_1)$ with
    \[
    I_1 = I_0 - \pi A, \qquad (g_1 -b_1) = (g_0-b_0) - AJ
    \]
    When there is no confusion, we will suppress the basepoint $m_0$ and identify the point $A\in\mc{GK}(\pi,J)$ with the underlying GK structures $(g,I,J,b)$.
\end{defn}

\begin{rmk} Following \cite{GualtieriBranes}, one way to construct a 2-form $A$ solving~\eqref{eq:def_gk_class} is to fix a 1-parameter family of functions $u_t$, generate a flow of $\pi$-Hamiltonian diffeomorphisms $\Phi_s$ by $X_{u_t}:=-\pi(du_t)$. Then setting 
    \[
    I_t=\Phi_t\cdot I_0, \quad A_t=\int_0^t dd^c_{I_s}u_sds
    \]
    we obtain a one-parameter family of exact forms $A_t$ which solve \eqref{eq:def_gk_class}. 
\end{rmk}
\begin{lemma}\label{lm:commutator}
    The tangent space space to $\mc{GK}(\pi,J)$ at a point corresponding to a GK structure $m=(g,I,J,b)$ is
    \[
    T_m\mc{GK}(\pi,J)\simeq \{\mbox{\rm exact } \dot{A}\in\Wedge^2(M) \,  \mbox{\rm of $I$-type (1,1)} \}.
    \]
    In particular there is a natural distribution $\mc{D}\subset T_m\mc{GK}(\pi,J)$ given by $dd^c_I$-exact 2-forms
    \[
    \mc{D}=\{dd^c_Iu\ |\ u\in C^\infty(M,\R) \}.
    \]
    This distribution is integrable with
    \[
        [\mathbf X_{u},\mathbf X_v]=-\mathbf X_{\pi(du,dv)},
    \]
    where $\mathbf X_{u}$ is a vector field $dd^c_Iu$ on $\GK(\pi,J)$ defined by $u\in C^\infty(M,\R)$.
\end{lemma}
\begin{proof} The proof is similar to the one of \cite{ASU2}, Lemma 2.15, by working with $A_t$ instead of $F_t$.
\end{proof}

\begin{lemma}\label{lm:gc_variation}
    Under the variation of $(g,I,J,b)$ induced by $dd^c_Iu\in \mathcal D\subset T_m\GK(\pi,J)$ we have
    \[ \dt g= -(dIdu)^{(1,1)}_J J, \qquad \dt I = -\pi(dIdu), \qquad \dt b = (dIdu)^{(2,0)+(0,2)}_J J,\]
    \[
    \dt\II= L^{H_0}_{\II Idu}\II, \qquad \dt\JJ=- L^{H_0}_{\JJ Idu}\JJ.
    \]
\end{lemma}
\begin{proof} This computation of the variation of $(g_t, I_t, b_t)$ follows from Definition~\ref{d:GK-class}.  From the variational formulae for $(g, I, b)$ we get
 \[\frac{d}{dt} \omega_J =(dIdu)^{(1,1)}_J, \qquad \frac{d}{dt} (bJ)= - (dIdu)^{(2,0) + (0,2)}_J, \]
 which yields 
 \[\frac{d}{dt} \left(d^c_J \omega_J - db\right)= -Jd(dIdu)=0,\]
 showing that 
 $H_0:=d^c_J\omega_J - db$ is constant.
 The variation of the GK data $(\II_t, \JJ_t)$ is analogous to the work of \cite{gibson2020deformation}, up to the interchanging the roles of $I$ and $J$.
\end{proof}

\begin{rmk}
    All of the above can be repeated verbatim whenever there is a compact Lie group $G\subset \mathrm{Diff}(M)$ and we are only considering $G$-invariant GK structures $(g,I,J,b)$ on $(M,H_0)$. Notice that by \eqref{eq:gk_torsion},  $H_0$ is then also $G$-invariant. In this case,  we get a notion of $G$-invariant generalized K\"ahler class which we will denote $\GK^{G}(\pi,J)$.  The corresponding distribution is $\mc{D}^G=\{dd^c_Iu\ |\ u\in C^\infty(M,\R)^{G}\}$.
\end{rmk}

\subsection{Complexified orbits}
We now consider an adapted  GK structure $(\II_0, \JJ_0) \in \mc{AGK}(\JJ_0, \vol_0)$. Let $\T\subset \mathrm{Diff}(M)$ be the torus generated by $X_0=\div_{\vol_0}\pi_{\JJ_0}$.  As in Remark \ref{r:divgremark}, we also assume that the corresponding bihermitian structure $m_0=(g_0, I_0, J, b_0)$ is $\T$-invariant.
We denote by $\mc F^\T\subset \GK_{m_0}^{\T}(\pi,J)$ a leaf of the $\T$-invariant distribution $\mc{D}^{\T}$.

\smallskip
We first establish the relation between $\mc F^\T$ and the moment map picture on $\mathcal{AGK}(\JJ,\vol)$ above.  We use  that the formal complex structure ${\bf J}$ on $\mc{AGK}(\JJ, \vol)$ (see \eqref{eq:formal_complex}) is integrable  and  $\mf{ham}(\JJ)$ preserves ${\bf J}$ in order to ``complexify'' the infinitesimal action of $\mf{ham}(\JJ )$ as follows: 
\begin{equation}\label{eq:complexified-action} \mf{ham}^\T(\JJ)\otimes {\mathbb C} \ni \Psi(\JJ(du)) + \i \Psi(\JJ(dv)) \to L_{\Psi(\JJ du)} \II + \II L_{\Psi(\JJ dv)} \II,\end{equation}
where the RHS is viewed as a vector field on $\mc{AGK}(\JJ, \vol)$. Of course,   it is not clear in the formal infinite dimensional setting that such a vector field admits a flow.

\begin{defn}[Complexified orbit] The vector fields on $\mc{AGK}(\JJ, \vol)$ of the form
\[\left\{ L^{H_0}_{\JJ du} \II +  \II L^{H_0}_{\JJ dv}\II, \, \, u, v\in C^{\infty}_0(M, \R)^{\T}\right\}\]
define a foliation on $\mc{AGK}(\JJ, \vol)$. A (formal) leaf of this foliation will be referred to as \emph{a complexified orbit} of $\mf{ham}^{\T}(\JJ)$. Thus, we say that a smooth path $\II_t \in \mc{AGK}(\JJ, \vol)$ belongs to the complexified orbit 
of $\mf{ham}^{\T}(\JJ)$ if $\frac{\partial}{\partial t} \II_t = L^{H_0}_{\JJ du_t} \II_t  +\II_t L^{H_0}_{\JJ dv_t} \II_t$ for some $\T$-invariant smooth functions $u_t, v_t$.
\end{defn}

We now show that a variation within the leaf $\mc F^\T$ corresponds naturally to a path in the complexified orbit of the infinitesimal action of $\mf{ham}^\T(\JJ_0)$ on $\mc{AGK}(\JJ_0, \vol_0)$.  In the K\"ahler setting,  this is achieved by fixing the K\"ahler form via Moser's Lemma, whereas here we use a one-parameter family of Courant automorphisms which fixes both $\JJ$ and $\vol$.

\begin{lemma}\label{l:moser}  Let $m_0=(g_0, I_0, J, b_0)$ be a bihermitan structure on $M$ corresponding to a GK structure $(\II_0, \JJ_0)$. Suppose that $\vol$ is adapted to $(g_0, I_0, J)$, generating a torus $\T$.  Let $m_t=(g_t,I_t,J,b_t)$ be a path in $\mc F^\T$ starting at $m_0$, corresponding to a family of functions $u_t\in C^\infty_0(M,\R)^\T$. Denote by $(\II_t,\JJ_t)$ the corresponding generalized complex structures and by  $\vol_t$ the volume forms obtained as  solution to
    \[
    \dt\vol_t =-L_{\pi_{\JJ_t}(I_tdu_t)}\vol_t, \qquad (\vol_t)_{|_{t=0}}= \vol_0.
    \]
    along $m_t$.   Let $\Phi_t\in\mathrm{Aut}(M,H_0)$, $\Phi_0=\mathrm{Id}$ be the isotopy corresponding to the time-dependent infinitesimal generator $- \JJ_t(I_t du_t)$, and by $\phi_t \in \Diff(M)$ the induced isotopy of diffeomorphisms, corresponding to the vector field $\pi_T(-\JJ_t(I_t du_t))= - \pi_{\JJ_t}(I_t du_t)$.  Then
    \[
    \Phi_t(\JJ_t)=\JJ_0,\quad (\phi_{t}^{-1})^*(\vol_t)=\vol_0.
    \]
    In particular, $\vol_t$ is $(g_t, I_t, J)$-adapted. Furthermore, the pullback data $\til{u}_t = (\phi_{t}^{-1})^* u_t$, $\til{\II}_t=\Phi_t(\II_t)$ and $\til{J}_t := \phi_t \cdot J$ satisfy 
    \[
    \dt\til {\II}_t=\til{\II}_t L^{H_0}_{\JJ_0 d\til u_t}\til\II_t, \qquad \dt \til{J}_t = \til{J}_t L_{\pi_{\JJ_0} d \til{u}_t} \til{J}_t.
    \]
    Furthermore, if $(\II_0, \JJ_0)$ is generically symplectic type, so is $(\II_t, \JJ_t)$ and $\tilde \II_t \in \mc{AGK}(\JJ_0, \vol_0)$ belongs to the complexified orbit of $\mf{ham}^\T(\JJ_0)$.
    
    \begin{proof} First note that the existence of $\Phi_t$ and $\phi_t$ follow from standard ODE methods.  In particular, the existence of $\phi_t$ is standard, and given this the putative $\Phi_t$ will satisfy
    \begin{align*}
        \dt (\phi_t^{-1} \circ \Phi_t) = \left(\Id - \pi_T \right) (\Psi( - \JJ_t I_t d u_t)),
    \end{align*}
    where $\phi_t^{-1}$ acts on $T \oplus T^*$ as usual.  In other words, after pulling back by $\phi_t^{-1}$ we now solve for a one-parameter family of Courant automorphisms covering the identity automorphism, in particular the family driven by the two-form component of $\Psi( - \JJ_t I_t d u_t)$.  Alternatively one can construct $\Phi_t$ directly using the interpretation of infinitesimal automorphisms as vector fields on $E$ (cf. \cite{SCSV} Proposition 2.7).
    
    The claims $\Phi_t(\JJ_t) = \JJ_0$ and $(\phi_{t}^{-1})^* (\vol_t) = \vol_0$ are immediate from the construction.  We also notice that $\Phi_t$ and $\phi_t$ are $\T$-equivariant.  To obtain the evolution equation for $\til{\II}$, it suffices to compute its derivative at $t = 0$.  Note that  using the Gualtieri map (\ref{eq:gualtieri_map}), it follows that
    \begin{equation}\label{eq:identity}
        (\II + \JJ) I du = \II \JJ du - du.
    \end{equation}
    Using the above relation, the integrability of $\til{\II}$, and the fact that the Lie derivative by an exact $1$-form vanishes, we obtain
    \begin{align*}
        \left. \dt \til{\II}_t\right|_{t=0} =&\ L^{H_0}_{\II_0 I_0 d u} \II_0 + L^{H_0}_{\JJ_0 I_0 du} \II_0 = L^{H_0}_{\II_0 \JJ_0 du - du} \II_0 = \II_0 L^{H_0}_{\JJ_0 du} \II_0,
    \end{align*}
    as claimed.  To obtain the evolution equation for $J$ we first note as an elementary consequence of (\ref{eq:gualtieri_map}) we have $\pi_{\JJ} I = J \pi_{\JJ}$ (cf. Lemma \ref{l:mabuchiidentities} below).  Using this we have
    \begin{align*}
        \left. \dt \til{J}_t \right|_{t=0} =&\ L_{\pi_{\JJ} I d u} J = L_{ J \pi_{\JJ} du} J = J L_{\pi_{\JJ} du} J.
    \end{align*}
    Finally, in the case $(\II_0, \JJ_0)$ is generically symplectic type, $\tilde \II_t \in \mc{AGK}(\JJ_0, \vol_0)=\mc{AGK}^{\T}(\JJ)$ as $(\tilde \II_t, \JJ_0)$ is $\T$-invariant, see Remark \ref{r:divgremark}.
    \end{proof}
\end{lemma}

\begin{lemma} \label{l:volumeform} Suppose $m_0=(g_0, I_0, J, b_0)$, $\vol_0$ and $\T \subset \Diff(M)$ are as in Lemma~\ref{l:moser}.  Suppose that $m_0$ is generically symplectic type. Let $m_t=(g_t,I_t,J,b_t)$ be a path in $\mc F^\T$ generated by a family of functions $u_t\in C^\infty_0(M,\R)^\T$.  Let $\vol_t$ be the family of volume forms along $m_t$ introduced in Lemma~\ref{l:moser}. Then $\vol_t$ depends only on the endpoint $m_t$, and does not depend on the specific path between $m_0$ and $m_t$.
\end{lemma}
\begin{proof}
    Let $m_t$, $t\in[0,1]$ be a closed loop in $\mc F^{\T}$. We need to prove that $\vol_0=\vol_1$. 
    By Lemmas~\ref{l:moser} and  \ref{l:invariant}, we know that along the path $m_t$,
    \[ X_t=\div_{\vol_t} \JJ_t = \phi_t^{-1} \cdot X_0 =X_0,\]
    (as $\phi_t$ preserves $\T$). Thus,
    \[
    \dt\div_\vol(\JJ)=0.
    \]
    If $f=\log\frac{\vol_1}{\vol_0}$, we use the definition of divergence and the fact that $df - \i \JJ_0 df \in \Ker (\psi)$, where $\psi$ is a defining spinor for $\JJ_0$, to obtain
    \[
    0=\left(\div_{\vol_1}(\JJ_0)-\div_{\vol_0}(\JJ_0)\right)=-\frac{1}{2}\JJ_0 (df),
    \]
    which implies that $f$ is constant.  On the other hand the total volume $\int_M\vol_t$ is fixed along $m_t$, proving that $f=1$ and $\vol_0=\vol_1$.
\end{proof}

\begin{rmk}
The above lemma proves that we can endow every point $m\in\mc F^\T$ with a volume form $\vol_m$.  When there is no confusion, we will suppress the subscript $m$ and simply denote this volume form by $\vol$ keeping in mind that it depends on a point in $\mc F^\T$.
\end{rmk}

\subsection{The Mabuchi metric}

Mabuchi \cite{Mabuchi} (cf. \cite{Semmes, Donaldson-git}) introduced a formal Riemannian structure on K\"ahler classes whose geometry captures fundamental features of the existence and uniqueness of cscK metrics.  In \cite{ASUScal} we extended this construction to generalized K\"ahler classes for structures of symplectic type.  Here we extend this to generalized K\"ahler structures  which are generically symplectic type, using the framework above.  Throughout the reminder of this section, we make the following

\begin{assumption}\label{a:generically-symplectic}  $m_0=(g_0, I_0, J, b_0)$ is a generically symplectic type bihermitian structure on $M$ which admits an adapted volume form $\vol_0$ generating a torus $\T$ of symmetries of $m_0$. We denote by $\mc{F}^{\T}=\mc{F}^{\T}_{m_0}$ the  corresponding generalized K\"ahler class of $\T$-invariant bihermitian structures.
By Lemma~\ref{l:moser}, any $m\in \mc{F}^{\T}_{m_0}$ is generically symplectic type and,   by Lemma~\ref{l:volumeform}, there is a well defined adapted volume form  $\vol_m$ associated to $m$.  We shall identify the tangent space at $m\in \mc{F}^{\T}$ (which is $C^{\infty}(M, \R)^{\T}/\R$) with the space of $\vol_m$-normalized $\T$-invariant smooth functions
\[ T_m\mc{F}^{\T} = C^{\infty}_{0}(M, \vol_m)=\left\{ u\in C^{\infty}(M, \R)^{\T} \, \, \Big| \, \, \int_M u \, \vol_m =0\right\}.\]
\end{assumption}

\begin{defn} The \emph{Mabuchi Riemannian metric} is defined for $m \in \mc F^{\T}$ and $u,v \in T_m \mc F^{\T}$ via
\begin{align*}
    \mab{u,v}_m = \int_M u v \vol_{m},
\end{align*}
where $\vol_m$ is the volume form guaranteed by Lemma \ref{l:volumeform}.
\end{defn}

The computations to follow on the geometry of this metric exploit a number of delicate identities for generalized K\"ahler structures, which we collect here:

\begin{lemma} \label{l:mabuchiidentities} The following identities hold:
\begin{align*}
    \pi_{\JJ}(\ga \wedge J \gb) =&\ - \pi_{\JJ} (I \ga \wedge \gb)\\
    \IP{\pi, \ga \wedge \gb} =&\ \pi_{\JJ} \left( (J - I) \ga \wedge \gb \right)\\
    \tr \left( \pi_{\JJ} (\gb \wedge I \gb) \circ \pi_{\JJ} (\ga \wedge J \ga) \right) =&\ \tfrac{1}{2} \left( \IP{ (I + J) \ga, \gb}_g^2 + \IP{(I + J) \ga, I \gb}_g^2 \right)\\
    \int_M u d (I dv \wedge i_{\pi_{\JJ}} \vol) =&\ \int_M v d ( J d u \wedge i_{\pi_{\JJ}} \vol).
\end{align*}
\begin{proof} The first is an easy consequence of the identity $I(I + J) = (I + J)J$.  The second follows directly from the definitions.  The third is a lengthty computation identical to \cite[(2.11)]{ASUScal}.  For the third we intgrate by parts and apply the first identity to yield
\begin{align*}
    \int_M u d (I dv \wedge i_{\pi_{\JJ}} \vol) =&\ - \int_M d u \wedge I dv \wedge i_{\pi_{\JJ}} \vol\\
    =&\ \int_M J du \wedge d v \wedge i_{\pi_{\JJ}} \vol\\
    =&\ \int_M v d ( J d u \wedge i_{\pi_{\JJ}} \vol),
\end{align*}
as claimed.
\end{proof}
\end{lemma}

\subsection{The Levi-Civita connection}

\begin{lemma} \label{l:LCformula} The Riemannian structure $\mab{\cdot,\cdot}$ admits a unique Levi-Civita connection $\mc D$, defined by
\begin{align*}
    \left( \mc D_{u} \bxi \right)_m = \dot{\bxi}_m(u) - \pi_{\JJ} (du \wedge J d \bxi_m),
\end{align*}
where $\dot{\bxi}_m(u) := \left. \dt \right|_{t=0} \bxi_{m_t}$, where $m_t$ is the canonical deformation of $m$ defined by $u$.
\begin{proof} We first compute the Levi-Civita connection for canonical vector fields.  In our formal setting, the Levi-Civita connection is still determined by the Koszul formula:
\begin{align*}
    2 \mab{\mc D_{\XX_{u_1}} \XX_{u_2}, \XX_{u_3}} =&\ \XX_{u_1} \mab{ \XX_{u_2}, \XX_{u_3}} + \XX_{u_2} \mab{\XX_{u_1}, \XX_{u_3}} - \XX_{u_3} \mab{\XX_{u_1}, \XX_{u_2}}\\
    &\ - \mab{\XX_{\{u_1,u_2\}_{\pi}}, \XX_{u_3}} + \mab{\XX_{\{u_2,u_3\}_{\pi}}, \XX_{u_1}} + \mab{\XX_{\{u_1,u_3\}_{\pi}}, \XX_{u_2}}.
\end{align*}
Next we observe using Lemma \ref{l:mabuchiidentities}:
\begin{align*}
  \XX_{u_1} \mab{ \XX_{u_2}, \XX_{u_3}} =&\ \int_M u_2 u_3 (- L^{H_0}_{\JJ I d u_1} \vol) = \int_M  u_2 u_3 (- L_{\pi_{\JJ} I d u_1} \vol) = \int_M u_2 u_3 d \left( I d u_1 \wedge i_{\pi_{\JJ}} \vol \right)\\
  \XX_{u_2} \mab{ \XX_{u_1}, \XX_{u_3}} =&\ \int_M u_1 u_3 d \left( I d u_2 \wedge i_{\pi_{\JJ}} \vol \right)\\
  \XX_{u_3} \mab{ \XX_{u_1}, \XX_{u_2}} =&\ \int_M u_1 u_2 d \left( I d u_3 \wedge i_{\pi_{\JJ}} \vol \right) = \int_M u_3 d (J d(u_1 u_2) i_{\pi_{\JJ}} \vol)\\
  =&\ \int_M u_3 u_1 d (J d u_2 \wedge i_{\pi_{\JJ}} \vol) + u_3 u_2 d (J d u_1 \wedge i_{\pi_{\JJ}} \vol) \\
  &\ \qquad + u_3 \left( d u_1 \wedge J d u_2 + d u_2 \wedge J d u_1 \right) \wedge i_{\pi_{\JJ}} \vol.
\end{align*}
Next we observe
\begin{align*}
    i_{\pi \ga} \vol = i_{ \pi_{\JJ} (J - I) \ga} \vol = (J - I) \ga \wedge i_{\pi_{\JJ}} \vol.
\end{align*}
Using this we compute
\begin{align*}
    \mab{\XX_{\{u_1,u_2\}_{\pi}}, \XX_{u_3}} =&\ \int_M u_3 \left( (J - I) d u_1 \wedge d u_2 \wedge i_{\pi_{\JJ}}\vol \right)\\
    \mab{\XX_{\{u_2,u_3\}_{\pi}}, \XX_{u_1}} =&\ \int_M u_1 \left( (J - I) d u_2 \wedge d u_3 \wedge i_{\pi_{\JJ}}\vol \right)\\
    =&\ \int_M u_3 \left( d u_1 \wedge (J - I) d u_2 \wedge i_{\pi_{\JJ}} \vol \right) + u_1 u_3 d \left( (J - I) d u_2 \wedge i_{\pi_{\JJ}} \vol \right)\\
    \mab{\XX_{\{u_1,u_3\}_{\pi}}, \XX_{u_2}} =&\ \int_M u_3 \left( d u_2 \wedge (J - I) d u_1 \wedge i_{\pi_{\JJ}} \vol \right) + u_2 u_3 d \left( (J - I) d u_1 ) \wedge i_{\pi_{\JJ}} \vol \right).
\end{align*}
Using these computations and Lemma \ref{l:mabuchiidentities} gives
\begin{align*}
    \left(\mathcal D_{\XX_{u_1}} \XX_{u_2} \right)_m =&\ - \pi_{\JJ} (d u_1 \wedge J d u_2) - \int_M u_2 d (I d u_1 \wedge i_{\pi_{\JJ}} \vol).
\end{align*}
Now for a general vector field $m \to \xi_m$ we compute
\begin{align*}
    \mab{\mathcal D_{u} \xi, \XX_{v}} =&\ \XX_{u} \mab{\xi, \XX_{v}} - \mab{\xi, \mathcal D_u \XX_{v}}\\
    =&\ \left. \dt \right|_{t=0} \int_M \xi v \vol + \int_M \xi (d u \wedge J d v \wedge i_{\pi_{\JJ}} \vol)\\
    =&\ \int_M \dot{\xi}(u) v \vol + \int_M \xi v d (I d u \wedge i_{\pi_{\JJ}} \vol) + \int_M \xi (d u \wedge J d v \wedge i_{\pi_{\JJ}} \vol)\\
    =&\ \int_M \left( \dot{\xi}(u) - \pi_{\JJ} (d u \wedge J d \xi) \right) v \vol,
\end{align*}
as claimed.
\end{proof}
\end{lemma}

\begin{defn} Let $m_t$ be a smooth path in $\mc F^{\T}$ corresponding to a Hamiltonian deformation with respect to a smooth path of functions $u_t \in C_0^{\infty}(M, \vol_t)$.  We say that $m_t$ is a \emph{geodesic} if
\begin{align*}
    0 = \mc D_{u} u = \dot{u} - \pi_{\JJ}(d u \wedge J du).
\end{align*}
\end{defn}

\begin{prop}\label{p:Killing-geodesic} Given the setup above, suppose $Y$ is a Hamiltonian vector field with respect to $\pi_{\JJ}$ which preserves $J$.
Then the flow $\Phi_t = \exp(-JY)$ of $-JY$ defines a geodesic in $\mathcal F^{\T}$.
\begin{proof} By the Hamiltonian assumption, there exists a smooth function $u_0$ such that
\begin{align*}
    Y = - \pi_{\JJ} d u_0, \qquad \int_M u_0 \vol = 0.
\end{align*}
By Lemma~\ref{l:hredinGKclass} below, the flow of $Y$ generates a path $m_t$ in $\mc{F}^{\T}$.

We let $u_t := \Phi_t^*(u_0)$.  Since $\Phi_t \cdot Y = Y$, it follows that $Y = - \pi_{\JJ_t} d u_t$.  Thus we compute
\begin{align*}
    \dt u_t = - L_{JY} u_t = d u_t(- JY) = \pi_{\JJ} (d u_t \wedge J du_t),
\end{align*}
as required.
\end{proof}
\end{prop}

\subsection{Curvature}

\begin{prop} \label{p:mabuchicurvature} At any given $m \in \mathcal{F}^{\mathbb T}$, the curvature tensor of the Mabuchi metric satisfies
\begin{align*}
    \left( \mathcal R(\XX_{u_1}, \XX_{u_2}) \XX_{u_3} \right)_m =&\ - \left\{ \left\{u_1, u_2\right\}_{\pi_{\JJ}}, u_3 \right\}_{\pi_{\JJ}}.
\end{align*}
In particular
\begin{align*}
    \mab{ \mathcal R(\XX_{u_1}, \XX_{u_2}) \XX_{u_1}, \XX_{u_2}}_m =&\ - \mab{ \{u_1, u_2\}_{\pi_{\JJ}}, \{u_1, u_2\}_{\pi_{\JJ}}},
\end{align*}
so that the sectional curvature is everywhere nonpositive.
    \begin{proof}
        We note that the second formula follows from the first using $\ad$-invariance of the metric (\ref{eq:ad-invariance}).  We will drop the basepoint from the computations to follow, and also ignore the additive normalizing constants of the underlying smooth functions, as these do not affect the answer.  Thus in particular we express
        \begin{align*}
            \mathcal D_{u} \XX_v =&\ - \tr_{\pi_{\JJ}} (d u \wedge J d v).
        \end{align*}
        Thus we first compute using Lemma \ref{l:LCformula},
        \begin{gather} \label{f:mabcurv10}
            \begin{split}
                \mathcal D_{u_1} \left( \mathcal D_{u_2} \XX_{u_3} \right) =&\ \tfrac{1}{2} \tr \left( \pi_{\JJ} d d^c_I u_1 \pi_{\JJ} (d u_2 \wedge J d u_3) \right) + \tr_{\pi_{\JJ}} (d u_1 \wedge J d \tr_{\pi_{\JJ}} (d u_2 \wedge J du_3))\\
                =&\ \tfrac{1}{2} \tr \left( \pi_{\JJ} d d^c_I u_1 \pi_{\JJ} (d u_2 \wedge J d u_3) \right) + \tr_{\pi_{\JJ}} (d \tr_{\pi_{\JJ}} (d u_2 \wedge J du_3) \wedge I d u_1).
            \end{split}
        \end{gather}
        Now we observe using Schur's Lemma that for $2$-forms $\ga, \gb$ one has
        \begin{gather} \label{f:mabcurv20}
           \ga \wedge \gb \wedge i_{\pi_{\JJ}} i_{\pi_{\JJ}} \vol = - \tfrac{1}{2} \tr (\pi_{\JJ} \ga \pi_{\JJ} \gb) \vol + (\tr_{\pi_{\JJ}} \ga) (\tr_{\pi_{\JJ}} \gb) \vol.
        \end{gather}
        Using this we obtain from (\ref{f:mabcurv10}) and integration by parts
        \begin{align*}
            \mab{ \mathcal D_{u_1} (\mathcal D_{u_2} \XX_{u_3}), \XX_{u_4}} =&\ - \int_M u_4 \left( d d^c_I u_1 \wedge d u_2 \wedge J d u_3 \right) \wedge i_{\pi_{\JJ}} i_{\pi_{\JJ}} \vol + \int_M u_4 \left(\tr_{\pi_{\JJ}} (d u_2 \wedge J d u_3) \right) d d^c_I u_1 \wedge i_{\pi_{\JJ}} \vol\\
            &\ + \int_M u_4 \left( d \tr_{\pi_{\JJ}} (d u_2 \wedge J du_3) \wedge I d u_1) \right) \wedge i_{\pi_{\JJ}} \vol\\
            =&\ \int_M d u_4 \wedge I d u_1 \wedge d u_2 \wedge J d u_3 \wedge i_{\pi_{\JJ}} i_{\pi_{\JJ}} \vol + \int_M u_4 \left( I d u_1 \wedge d u_2 \wedge d J d u_3 \right) \wedge i_{\pi_{\JJ}} i_{\pi_{\JJ}} \vol\\
            &\ - \int_M u_4 (I d u_1 \wedge d u_2 \wedge J d u_3) \wedge d i_{\pi_{\JJ}} i_{\pi_{\JJ}} \vol\\
            &\ - \int_M \tr_{\pi_{\JJ}} ( d u_2 \wedge J d u_3) \tr_{\pi_{\JJ}} (d u_4 \wedge I d u_1) \vol + \int_M u_4 \tr_{\pi_{\JJ}} (d u_2 \wedge J du_3) I d u_1 \wedge d i_{\pi_{\JJ}} \vol.
        \end{align*}
        Also we compute
        \begin{align*}
            & \mab{ \mathcal D_{\{u_1, u_2\}_{\pi}} \XX_{u_1}, \XX_{u_2}} = - \int_M u_2 \left( d \{u_1, u_2\}_{\pi} \wedge J d u_1 \right) \wedge i_{\pi_{\JJ}} \vol\\
            &\ \quad = \int_M \{u_1, u_2\}_{\pi} d u_2 \wedge J d u_1 \wedge i_{\pi_{\JJ}} \vol + \int_M u_2 \{u_1, u_2\}_{\pi} d d^c_J u_1 \wedge i_{\pi_{\JJ}} \vol - \int_M u_2 \{u_1, u_2\} J d u_1 \wedge d i_{\pi_{\JJ}} \vol\\
            &\ \quad = - \int_M \tr_{\pi_{\JJ}} ( d u_1 \wedge (I - J) d u_2) \tr_{\pi_{\JJ}} (d u_2 \wedge J d u_1 ) \vol - \int_M u_2 \tr_{\pi_{\JJ}} ( d u_1 \wedge (I - J) d u_2) \tr_{\pi_{\JJ}} d d^c_J u_1 \vol\\
            &\ \qquad + \int_M u_2 \tr_{\pi_{\JJ}} ( d u_1 \wedge (I - J) d u_2) J d u_1 \wedge d i_{\pi_{\JJ}} \vol\\
            &\ \quad = - \int_M \left( \tr_{\pi_{\JJ}} \left( d u_2 \wedge J d u_1 \right) - \tr_{\pi_{\JJ}} (d u_2 \wedge I d u_1) \right) \tr_{\pi_{\JJ}} (d u_2 \wedge J d u_1) \vol\\
            &\ \qquad - \int_M u_2 \left( \tr_{\pi_{\JJ}} (d u_1 \wedge I du_2) + \tr_{\pi_{\JJ}} (I d u_1 \wedge d u_2) \right) \tr_{\pi_{\JJ}} d d^c_J u_1 \vol\\
            &\ \qquad + \int_M u_2 \tr_{\pi_{\JJ}} ( d u_1 \wedge (I - J) d u_2) J d u_1 \wedge d i_{\pi_{\JJ}} \vol.
        \end{align*}
        We divide the terms in the expression for curvature into those which do not involve a derivative of the volume form ($T_1$) and those which do ($T_2$).  Using the computations above the former comprise:
        \begin{align*}
            T_1 =&\ \int_M (\tr_{\pi_{\JJ}} (d u_2 \wedge J d u_1)^2 \vol\\
            &\ + \int_M (d u_2 \wedge I d u_2 \wedge d u_1 \wedge J d u_1) \wedge i_{\pi_{\JJ}} i_{\pi_{\JJ}} \vol - \int_M (\tr_{\pi_{\JJ}} ( du_1 \wedge J du_1) (\tr_{\pi_{\JJ}}(d u_2 \wedge I du_2)) \vol\\
            &\ - \int_M u_2 \left( d u_1 \wedge I du_2 \wedge d d^c_J u_1 \right) \wedge i_{\pi_{\JJ}} i_{\pi_{\JJ}} \vol + \int_M u_2 \left( \tr_{\pi_{\JJ}} (d u_1 \wedge I du_2) \tr_{\pi_{\JJ}} (d d^c_J u_1) \right) \vol\\
            &\ - \int_M u_2 (I d u_1 \wedge d u_2 \wedge d d^c_J u_1) \wedge i_{\pi_{\JJ}} i_{\pi_{\JJ}} \vol + \int_M u_2 \left( \tr_{\pi_{\JJ}} (I d u_1 \wedge d u_2) \tr_{\pi_{\JJ}} ( d d^c_J u_1) \right) \vol
        \end{align*}
        Using (\ref{f:mabcurv20}) and Lemma \ref{l:mabuchiidentities} we simplify the final four terms as:
        \begin{align*}
            \tfrac{1}{2} & \int_M \left[ \tr \left( \pi_{\JJ} (d u_1 \wedge I d u_2) \pi_{\JJ} d d^c_J u_1 \right) + \tr \left( \pi_{\JJ} (I d u_1 \wedge d u_2)  \pi_{\JJ} d d^c_J u_1 \right) \right] \vol\\
            =&\ \tfrac{1}{2} \int_M \left[ \tr \left( \pi_{\JJ} (d u_1 \wedge I d u_2) \pi_{\JJ} d d^c_J u_1 \right) + \tr \left( \pi_{\JJ} (I d u_1 \wedge d u_2)  \pi_{\JJ} J d d^c_J u_1 J \right) \right] \vol\\
            =&\ \tfrac{1}{2} \int_M \left[ \tr \left( \pi_{\JJ} (d u_1 \wedge I d u_2) \pi_{\JJ} d d^c_J u_1 \right) + \tr \left( \pi_{\JJ} (II d u_1 \wedge I d u_2)  \pi_{\JJ} d d^c_J u_1 \right) \right] \vol\\
            =&\ 0.
        \end{align*}
        With similar simplifications we finally yield
        \begin{align*}
            T_1 =&\ \int_M (\tr_{\pi_{\JJ}} (d u_2 \wedge J d u_1)^2 \vol\\
            &\ + \int_M (d u_2 \wedge I d u_2 \wedge d u_1 \wedge J d u_1) \wedge i_{\pi_{\JJ}} i_{\pi_{\JJ}} \vol - \int_M (\tr_{\pi_{\JJ}} ( du_1 \wedge J du_1) (\tr_{\pi_{\JJ}}(d u_2 \wedge I du_2)) \vol\\
            =&\ \int_M (\tr_{\pi_{\JJ}} (d u_2 \wedge J d u_1)^2 \vol - \tfrac{1}{2} \int_M \tr \left( \pi_{\JJ} (d u_2 \wedge I du_2) \pi_{\JJ} (d u_1 \wedge J d u_1) \right) \vol\\
            =&\ \int_M (\tr_{\pi_{\JJ}} (d u_2 \wedge J d u_1)^2 \vol - \int_M ( \tr_{\pi_{\JJ}} (d u_1 \wedge d u_2))^2 \vol - \int_M ( \tr_{\pi_{\JJ}} (d u_1 \wedge I du_2))^2 \vol\\
            =&\ - \int_M ( \tr_{\pi_{\JJ}} (d u_1 \wedge d u_2))^2 \vol.
        \end{align*}
        It remains to simplify the terms involving the exterior derivative of the volume form.  To that end we first note that
        \begin{align*}
    d i_{\pi_{\JJ}} i_{\pi_{\JJ}} \vol = i_{\pi_{\JJ}} i_X \vol.
\end{align*}
Using this we see
\begin{align*}
    T_2 =&\ \int_M \left( I d u_1 \wedge d u_2 \wedge J du_1 - I du_2 \wedge d u_1 \wedge J du_1 \right) \wedge (i_{\pi_{\JJ}} i_X \vol)\\
    &\ \qquad + \left( \tr_{\pi_{\JJ}} (d u_1 \wedge J du_1) I du_2 - \tr_{\pi_{\JJ}} (d u_2 \wedge J du_1) I du_1 - \tr_{\pi_{\JJ}} (d u_1 \wedge (I - J) d u_2) J du_1 \right) \wedge i_X \vol.
\end{align*}
Next we claim the identity:
\begin{align*}
    \ga_1 \wedge \ga_2 \wedge \ga_3 \wedge i_{\pi_{\JJ}} i_X \vol = \left( \tr_{\pi_{\JJ}}( \ga_1 \wedge \ga_2) \ga_3 - \tr_{\pi_{\JJ}} (\ga_1 \wedge \ga_3) \ga_2 + \tr_{\pi_{\JJ}} (\ga_2 \wedge \ga_3) \ga_1 \right) \wedge i_X \vol.
\end{align*}
Furthermore we note by the $X$ invariance of $u_i$ it follows that $d u_i \wedge i_X \vol = X u_i \vol = 0$.
Using these points we observe
\begin{align*}
    T_2 =&\ \int_M \left(  (\tr_{\pi_{\JJ}} (d u_2 \wedge J d u_1) - \tr_{\pi_{\JJ}} (d u_2 \wedge J du_1) ) I d u_1 + (\tr_{\pi_{\JJ}} (d u_1 \wedge J du_1) - \tr_{\pi_{\JJ}} (d u_1 \wedge J du_1)) I d u_2 \right.\\
    &\ \qquad \left. + ( \tr_{\pi_{\JJ}} (I d u_1 \wedge d u_2) - \tr_{\pi_{\JJ}} (I du_2 \wedge d u_1) - \tr_{\pi_{\JJ}} (d u_1 \wedge (I - J) d u_2 )) J d u_1 \right) \wedge i_X \vol\\
    =&\ 0.
\end{align*}
    \end{proof}
\end{prop}

\section{The generalized K\"ahler Calabi problem.}

In this section, we combine the moment map setup with the Mabuchi geometry of $\mc{F}^{\T}$ developed in Section~\ref{s:scalasmoment} in order to obtain results about the existence of extremal bihermitian structures $m\in \mc{F}^{\T}$, see Definition~\ref{d:extremal-BH}.  We still assume throughout this section that Assumption~\ref{a:generically-symplectic} holds true.

\subsection{Calabi Energy}

\begin{defn} \label{d:Calabi} Given a generalized K\"ahler class $\mathcal {F}^{\T}$, we define the \emph{Calabi functional} as
\begin{align*}
    \mathbf{Ca} :&\ \mathcal {GK}(\pi,J)^{\T} \to \mathbb R\\
    \mathbf{Ca} (\II, \JJ,\vol) =&\ \int_M \Gscal (\II, \JJ, \vol)^2 \vol.
\end{align*}
\end{defn}

\begin{prop} \label{p:Calabivar} $(\II, \JJ,\vol) \in \mathcal {GK}(\pi,J)^{\T}$ is a critical point of $\mathbf{Ca}$ if and only if $(\II, \JJ, \vol)$ is extremal.
\begin{proof} Using the Moser isotopy $\Phi_t$ from Lemma \ref{l:moser} we can compute for an arbitrary variation in $\mathcal {F}^{\T}$ as in Corollary \ref{cor:extremal}
\begin{align*}
    \left. \frac{d}{dt} \right|_{t=0} \mathbf{Ca}(\II_t, \JJ_t,\vol_t) =&\ \left. \frac{d}{dt} \right|_{t=0}  \Phi_t^* \mathbf{Ca}(\II_t, \JJ_t,\vol_t)\\
    =&\ \left. \frac{d}{dt} \right|_{t=0} \int_M \Gscal^2(\II_t, \JJ, \vol) \vol\\
    =&\ \pmb{\Omega} \left( L_{\chi}\II, \II L_{\Psi(\JJ d u_0)} \II \right).
\end{align*}
Specializing to the case $u_0 = \Gscal(\II, \JJ, \vol)$ it follows that $\brs{ L_{\chi}\II}^2_{\pmb{g}} = 0$, giving the claim.
\end{proof}
\end{prop}

\subsection{Mabuchi \texorpdfstring{$1$}{1}-form}

\begin{defn}
    The Mabuchi 1-form $\pmb{\tau}$ on $\mc F^\T$ is defined through its values on the fundamental vector fields $\mathbf{X}_u$, $u\in C_0^\infty(M,m)^\T$:
    \[
    \pmb{\tau}(\mathbf{X}_u):=-\int_M u \Gscal_\vol(\II,\JJ)\vol.
    \]
\end{defn}

\begin{prop} \label{p:Mabuchiclosed}
    The Mabuchi 1-form $\pmb{\tau}$ is closed.
\end{prop}
\begin{proof}
    We will verify that
    \[
    \mathbf{X}_u\cdot\pmb{\tau}(\mathbf{X}_v)-\mathbf{X}_v\cdot\pmb{\tau}(\mathbf{X}_u)+\pmb{\tau}(\mathbf{X}_{\pi(du,dv)})=0,
    \]
    where we used the identity for $[\mathbf X_u,\mathbf X_v]$ from Lemma~\ref{lm:commutator}.  Let us start with calculating the first term. We denote by $\left. \dt \right|_{t=0}$ the infinitesimal action of $\mathbf{X}_u$
    \begin{equation}
        \begin{split}
            \mathbf{X}_u\cdot\pmb{\tau}(\mathbf{X}_v)=-\left. \frac{d}{dt} \right|_{t=0}\int_M v \Gscal(\II_t,\JJ_t,\vol_t)\vol_t
        \end{split}
    \end{equation}
    Pulling back by a generalized automorphism generated by $\JJ(I du)$, and denoting by $\til\II_t$ the regauged one-parameter family of generalized complex structures with
    \[
    \left. \dt\right|_{t=0}\til \II_t=\II L^{H_0}_{\JJ du}\II,
    \]
    we can rewrite the last expression as
    \[
    \begin{split}
        \mathbf{X}_u\cdot\pmb{\tau}(\mathbf{X}_v)&=
        -\int_M v \left. \dt\ \right|_{t=0}\Gscal(\til\II_t, \JJ,\vol)\vol-
        \int_M \la dv, {\JJ Idu}\ra\Gscal(\II, \JJ,\vol)\vol\\
        &=\pmb\Omega(\II L^{H_0}_{\JJ du}\II, L^{H_0}_{\JJ dv} \II)+\frac{1}{2}\int_M \la dv, g^{-1}(I+J)Idu\ra \Gscal(\II, \JJ,\vol)\vol.
    \end{split}
    \]
    The first term of the last expression is symmetric in $u$ and $v$ since $\pmb \Omega$ is an $\II$-invariant 2-form. The anti-symmetrization of the last expression is exactly
    \[
    \int_M \pi(du,dv) \Gscal_\vol(\II,\JJ)\vol=-\pmb\tau(\mathbf X_{\pi(du,dv)}),
    \]
    proving that $\pmb\tau$ is closed.
\end{proof}

\begin{prop} \label{p:Mabuchiconvex} Let $m_t$ be a smooth geodesic in $\mc F^{\T}$, corresponding to $u_t$.  Then
\begin{align*}
    \frac{d}{dt} \pmb{\tau} (\XX_{u_t}) \geq 0,
\end{align*}
with equality if and only if $u_t$ is induced by a $\pi_{\JJ}$-Hamiltonian vector  field  as in Proposition~\ref{p:Killing-geodesic}. 
\begin{proof} We let $u_t$ denote a geodesic, and observe that it suffices to compute
    \begin{align*}
        - \frac{d}{dt} \int_M u_t \Gscal_{\vol_t} (\II_t, \JJ_t) \vol_t
    \end{align*}
at $t = 0$.  We construct the Moser isotopy $\Phi_t$ from Lemma \ref{l:moser}.  We first observe that for a geodesic, the pullback potentials $\til{u}_t = \Phi_t^* u_t$ are constant:
\begin{align*}
    \frac{d}{dt} \til{u}_t =&\ \Phi_t^* \left( L_{Z_t} u_t + \dot{u}_t \right)\\
    =&\ \Phi_t^* \left( \pi_{\JJ_t} (I_t d u_t, d u_t) + \dot{u}_t \right)\\
    =&\ \Phi_t^* \left( - \pi_{\JJ_t} (d u_t, J d u_t) + \dot{u}_t \right)\\
    =&\ 0.
\end{align*}
Using this, Theorem \ref{t:gotomoment} and Lemma \ref{l:moser} we have
\begin{align*}
    - \left. \frac{d}{dt} \right|_{t=0} \int_M u_t \Gscal_{\vol_t}(\II_t, \JJ_t) \vol_t =&\ - \left. \frac{d}{dt} \right|_{t=0} \int_M \Phi_t^* \left( u_t \Gscal_{\vol_t}(\II_t, \JJ_t) \vol_t \right)\\
    =&\ - \left. \frac{d}{dt} \right|_{t=0} \int_M u_0 \Gscal_{\vol}(\II_t, \JJ) \vol\\
    =&\ - \pmb{\Omega}(L^{H_0}_{\JJ d u_0}\II,\dot{\II})\\
    =&\ \pmb{\Omega}(L^{H_0}_{\JJ d u_0}\II,\II L^{H_0}_{\JJ d u_0}\II)\\
    =&\ \brs{\brs{ L^{H_0}_{\JJ d u_0}\II}}_{\pmb{g}}.
\end{align*}
This proves the claimed monotonicity.  In the case of equality it follows that, along the original geodesic, the section $e_t := \JJ_t d u_t$ preserves $\II_t$ for all $t$.  Since it also preserves $\JJ_t$ in general, it follows that $Y_t:=\pi_T(e_t)=\pi_{\JJ_t}(du_t)$ preserves the bihermitian data $(g_t, I_t, J)$ as well.  We finally claim that $Y_t$ is in fact constant in time.  To this end we compute, using the geodesic equation expressed as $\dot{u} = i_Y I d u$
\begin{align*}
    \dt Y_t =&\ \dt \pi_{\JJ_t} d u_t = \pi_{\JJ_t} \left( i_Y d d^c_I u + d \dot{u} \right)\\
    =&\ \pi_{\JJ} ( i_Y d d^c_I u + d i_Y d^c_I u)\\
    =&\ \pi_{\JJ} L_Y (I d u)\\
    =&\ \pi_{\JJ} I d L_Y u\\
    =&\ 0.
\end{align*}
Thus, $u_t$ is given by Proposition~\ref{p:Killing-geodesic}. \end{proof}
\end{prop}

As a consequence of the above proposition we obtain uniqueness of constant scalar curvature GK structures, predicated upon the existence of a smooth geodesic connecting and two cscGK structures.

\begin{cor}\label{c:conditional-uniqueness} Suppose $m_0, m_1\in \mc{F}^{\T}$ are cscGK structures which can be joined by a smooth geodesic $m_t$.  Then $m_1$ and $m_2$ are isometric.
\end{cor}

\section{The Calabi-Licherowicz-Matsushima obstruction and the Futaki invariant} \label{s:matsu}
\subsection{The formal statement}

\begin{defn}\label{d:hamiltonian_isometry}
We say that an element $\JJ du \in  \mf{ham}^\T(\JJ)$ is a \emph{Hamiltonian isometry} of  a $(\JJ, \vol)$-adapted generalized K\"ahler structure $(\II, \JJ)$ if
$L^{H_0}_{\JJ du} \II =0$. We denote by
\[ \mf{g}^\T(\II, \JJ) < \mf{ham}^\T(\J)\]
the sub-algebra of Hamiltonian isometries of $(\II, \JJ)$. If $m=(g, I, J, b)$ is the corresponding biHermitian structure, we also consider the Lie algebra of vector fields 
\[\mf{g}^{\T}(m)=\{ V= \pi_{\JJ}(du), \, u\in C^{\infty}_0(M, \R)^{\T}\, | \, L_Vg=0, L_VI=L_VJ=0\}.\]\end{defn}

\begin{defn} If $\II$ is a  generalized K\"ahler structure in $\mc{AGK}(\JJ, \vol)$, we denote by
\[ \mf{h}^{\T}(\II, \JJ) := \left\{\Psi(\JJ du)+ \i \Psi(\JJ dv) \in \mf{ham}^\T(\JJ)\otimes {\mathbb C} \, \,  | \,  \,  L^{H_0}_{\JJ du} \II + \II L^{H_0}_{\JJ dv} \II=0\right\} \subset \mf{ham}^\T(\JJ)\otimes {\mathbb C}.\]
Using the integrability of $\II$, we identify
\[ \mf{h}^{\T}(\II, \JJ) = \left\{\Psi(\JJ du)+ \i \Psi(\JJ dv) \in \mf{ham}^\T(\JJ)\otimes {\mathbb C} \, \,  | \,  \,  L^{H_0}_{\JJ du + \II \JJ dv} \II=0\right\},\]
which shows that  $\mf{h}^{\T}(\II, \JJ) < \mf{ham}^\T(\JJ)\otimes {\mathbb C}$ is then the subalgebra identified with 
the stabilizer of $\II$ inside $\mf{ham}^\T(\JJ)\otimes {\mathbb C}$.
We  shall refer to  $\mf{h}^{\T}(\II, \JJ)$ as the Lie algebra of \emph{reduced infinitesimal generalized complex automorphisms} of $(\II, \JJ)$. If $\chi \in \mf{h}^{\T}(\II, \JJ)$, we denote by $\left(\mf{h}^{\T}(\II, \JJ)\right)^{\chi}< \mf{h}^{\T}(\II, \JJ)$ the commutator of $\chi$ inside $\mf{h}^{\T}(\II, \JJ)$.

\end{defn}
By the definitions of $\mf{g}(\II, \JJ)$ and $\mf{h}(\II, \JJ)$, we have the inclusion
\[ \mf{g}^{\T}(\II, \JJ)\otimes \C < \mf{h}^{\T}(\II, \JJ).\]
Using the formal moment map picture from the previous section,  we deduce  (similarly to  \cite{Goto_2020}) the following generalized Calabi-Lichnerowicz-Matsushima theorem.

\begin{thm}\cite{Goto_LM}[Calabi-Lichnerowicz-Matsushima]\label{thm:CML} Suppose $(\II, \JJ, \vol)$ is a $\T$-invariant  extremal GK data on $(M, H_0)$ with extremal Hamiltonian isometry $\chi=\Psi(\JJ d \Gscal(\II, \JJ, \vol))$. Then 
\[ \left(\mf{h}^{\T}(\II, \JJ)\right)^{\chi} =\mf{g}^{\T}(\II, \JJ)\otimes \C.\]
\end{thm}
\begin{proof} This follows from the general  statement in \cite{LSW}, see also Theorem~\ref{thm:CML-decomp} for a short proof.
\end{proof} 

\subsection{Reductive Lie algebras}\label{s:red}

We would like to use Theorem~\ref{thm:CML} in order to deduce computable obstructions for the existence of extremal GK structures in terms of automorphisms group. To this end, we are going to recast the Lie algebras $\mf{g}^{\T}(\II, \JJ)$  and $\mf{h}^{\T}(\II, \JJ)$ in terms of the corresponding bihermitian data.
\begin{defn} Let $m=(g, I, J, b)$ be the  bihermitian structure corresponding to $(\II, \JJ)$.  Define the Lie algebra of vector fields 
\[\mf{g}^{\T}(m)=\{ V= \pi_{\JJ}(du), \, u\in C^{\infty}_0(M, \R)^{\T}\, | \, L_Vg=0, L_VI=L_VJ=0\}.\]
\end{defn}
\begin{lemma}\label{l:hamiltonian_Killing}
The vector field $V: = \pi_T(\JJ du)=\poiss_{\JJ}(du)$ associated to a Hamiltonian isometry $\JJ(du)$ of $(\II, \JJ)$ preserves the corresponding bihermitian 
structure $(g, I, J)$ and volume form $\vol$. In particular,  
\[\pi_T(\mf{g}^\T(\II, \JJ)) < \mf{g}^{\T}(m) \]
is a compact sub-algebra of vector fields. 
\end{lemma}
\begin{proof} This follows from Lemma~\ref{l:hamiltonians}.
\end{proof}
In the generically symplectic type case the elements of $\mf{g}^\T(m)$ canonically lift to elements of $\mf{g}^\T(\II, \JJ)$:
\begin{lemma}\label{l:hamiltonian_Killing_converse} Let $m=(g, I, J, b)$ be a generalized K\"ahler structure of generically symplectic type.  Suppose $\vol$ is a $(g,I, J)$ adapted volume form, generating a torus $\T$, and suppose $b$ is $\T$-invariant.  Let $V\in \mf{g}^{\T}(m)$ with potential $u\in C_0^{\infty}(M, \R)^{\T}$. 
Then $L^{H_0}_{\JJ du} \JJ = L^{H_0}_{\JJ du} \II =0$, i.e. $\Psi(\JJ du) \in \mf{g}^{\T}(\II, \JJ)$.
\end{lemma}
\begin{proof} By Theorem~\ref{thm:gualtieri_map} we have 
\[ L_V \II = L_{V} \JJ=0, \qquad \JJ du= V + I du - Jdu + b(V, \cdot). \]
Writing $\Psi(\JJ du) = V + B$ with $B = dd^c_I u - dd^c_J u + d(\imath_V b) $, we know by Lemma~\ref{l:hamiltonians} and the $V$-invariance of the bihermitian data
\[ 0=L^{H_0}_{\JJ du} \JJ = L_{V + B} \JJ= L_B \JJ. \]
As $\pi_{\JJ}$ is non-degenerate almost everywhere, we have $B=0$. It then follows that
\[ L^{H_0}_{\JJ du} \II = L_V \II =0.\]
\end{proof}

\begin{defn} Given a $\T$-invariant generalized K\"ahler structure $m=(g, I, J,b)$ on a compact manifold $M$,  with associated Poisson tensor $\pi_{\JJ}:=\frac{1}{2}(I+J)g^{-1}$  we let
\[\mf{h}_{\rm red}^{\T}(m):=\left\{ V = \poiss_{\JJ}(du) + J\poiss_{\J}(dv), \, u,v \in C^{\infty}(M, \R)^{\T} \, \, | \,  \,
L_V J = 0\right\}. \]
Notice that
\[\mf{g}^{\T}(m) + J \mf{g}^{\T}(m) < \mf{h}_{\rm red}^{\T}(m).\]
\end{defn}

\begin{lemma}\label{l:reduced} $\mf{h}_{\rm red}^{\T}(m)$ is a complex Lie sub-algebra of the Lie algebra $\mf{aut}^{\T}(M, \poiss, J)$  of real holomorphic  vector fields on $(M, J)$ which are $\T$-invariant and preserve the holomorphic Poisson tensor $\pi-\sqrt{-1}J\pi$.  The complex structure of $\mf{aut}^{\T}(M, \poiss, J)$  is induced by the action of $J$ on  real holomorphic vector fields.
\end{lemma}
\begin{proof} The vector space $\mf{h}_{\rm red}^{\T}(m)$ is clearly $J$-invariant. We check that $\mf{h}_{\rm red}^{\T}(m)$ is a Lie algebra.
Let
\[ V_1= \poiss_{\JJ}(du_1) + J \pi_{\JJ}(dv_1), \qquad V_2=\poiss_{\JJ}(du_2) + J \pi_{\JJ}(dv_2)\] be two elements of $\mf{h}^{\T}_{\rm red}(m)$.  In the computation below, we use that $L_{V_i} J=0$, the Poisson property $[\pi_{\JJ}(du), \pi_{\JJ}(dv)]= 
\pi_{\JJ}(d\{u, v\}_{\pi_{\JJ}})$, and the integrability condition $J(L_V J)= -L_{JV}J$:
\[
\begin{split}
[V_1, V_2]= & \, \pi_{\JJ}(d\{u_1, u_2\}_{\pi_{\JJ}}) + J\pi_{\JJ}(d\{u_1, v_2\}_{\pi_{\JJ}}) - [\pi_{\JJ}(du_2), J \pi_{\JJ}(dv_1)]  - J[\pi_{\JJ}(dv_2), J\pi_{\JJ}(dv_1)]
\\
= & \,  \pi_{\JJ}\left(d[\{u_1, u_2\}_{\pi_{\JJ}}-\{v_1, v_2\}_{\pi_{\JJ}}]\right) + J\pi_{\JJ}\left(d[\{u_1, v_2\}_{\pi_{\JJ}}\ +\{v_1, u_2\}_{\pi_{\JJ}}]\right) -  (L_{V_2} J)(\pi_{\JJ}(dv_1)) \\ =& \pi_{\JJ}\left(d[\{u_1, u_2\}_{\pi_{\JJ}}-\{v_1, v_2\}_{\pi_{\JJ}}]\right) + J\pi_{\JJ}\left(d[\{u_1, v_2\}_{\pi_{\JJ}}\ +\{v_1, u_2\}_{\pi_{\JJ}}]\right).
\end{split} \]
We now show that the elements of $\mf{h}_{\rm red}^{\T}(m)$ preserve $\pi$. By the Poisson version of Cartan's formula, we have
\begin{equation}\label{poiss-variation-JJ} L_V \poiss_{\JJ} = \poiss_{\JJ} (dI dv) \poiss_{\JJ}.\end{equation}
From the expression of $\pi_{\JJ}$ and the definition \eqref{eq:Poisson_BH} of $\poiss$, we have 
\begin{equation}\label{poiss-JJ-poiss}
J \poiss_{\JJ} + \poiss_{\JJ} J^* =  (J-I) \poiss_{\JJ} =-\poiss.\end{equation}
Using $L_V J=0$ (which holds by assumption) and \eqref{poiss-variation-JJ}, the above  equality gives
\begin{equation}\label{poisson-variation}
\begin{split}
L_V \poiss &= -J\poiss_{\JJ} (dIdv) \poiss_{\JJ} -\poiss_{\JJ} (dIdv) \poiss_{\JJ} J^* \\
&=\poiss_{\JJ}\left(I^* (dIdv) +(dIdv)I\right)\poiss_{\JJ} \\
&=0,
\end{split}
\end{equation}
where we have used that $\poiss_{\JJ} J^* = - I \poiss_{\JJ}, \, J \poiss_{\JJ} = -\poiss_{\JJ} I^*$ to obtain the second line and $I^*(dIdv) + (dIdv)I=0$ (by type considerations with respect to $I$) to get the third line.  This proves the inclusion $ \mf{h}_{\rm red}^{\T}(m) < \mf{aut}^{\T}(M, \poiss, J)$. \end{proof}

{
\begin{lemma} \label{l:hredinGKclass} Let $m_0=(g_0, J, I_0, b_0)$ be a generically symplectic type bihermitian  structure and $\T$ be a compact torus of isometries of $m_0$.  Then, for any $V\in \mf{h}^{\T}_{\rm red}(m_0)$, the flow $\phi_t$ of $V$ defines tensors $g_t=(\phi_{t}^{-1})^{*}g_0, \,  J= (\phi_{t}^{-1})_*J (\phi_{t})_*=J, \, I_t=(\phi_{t}^{-1})_*I_0 (\phi_{t})_*$ which give rise to a path  $m_t=(g_t, J, I_t, b_t) \in \mc{F}^{\T}_{m_0}$ for a suitable choice of $\T$-invariant $2$-forms $b_t$.
\end{lemma}
\begin{proof}
Denote by $\phi_t$ the flow of $V$ and let  $m_t : = \phi_{t} \cdot m$ be the variation of the bihermitian data under this flow. Clearly, $J_t=J$  as the flow of $V$ preserves $J$, and 
the corresponding  Poisson tensor is $\poiss_t=(\phi_{ t})_*(\poiss)=\poiss$ because of \eqref{poisson-variation}. As $(\phi_{t})_* \pi_{\JJ}= \pi_{\JJ_t}$, we have
\[V =\poiss_{\JJ_t}(du_t) + J\poiss_{\JJ_t}(dv_t) \]
with $u_t = \phi_{t}^*(u)$ and $v_t := \phi_{t}^*(v)$. 
From \eqref{poiss-variation-JJ}, \eqref{poiss-JJ-poiss} and $L_V J=0$
we get 
\[ 0=L_V \poiss = L_V\left((I-J) \poiss_{\JJ}\right)=(L_V I) \poiss_{\JJ} + (I-J) \poiss_{\JJ}  (dIdv)  \poiss_{\JJ} = \Big((L_V I)+ \poiss(dIdv)\Big)  \poiss_{\JJ}. \]
The fact that $\poiss_{\JJ}$ is invertible on an open dense subset yields that
\begin{equation}\label{I-variation}
   L_V I=- \poiss(dIdv) 
\end{equation} 
on this subset, and hence, by continuity,  everywhere on $M$. Similarly, using \eqref{poiss-variation-JJ} and $L_V J=0$, we compute 
\[ 
\begin{split}
L_V I &=L_V(I+J)=2L_V(\poiss_{\JJ}g) =2\poiss_{\JJ} (dIdv) \poiss_{\JJ} g +2\poiss_\JJ (L_V g) \\
&= \poiss_{\JJ} (dIdv) (I+J) + 2\poiss_{\JJ}(L_V g) \\
&=-\poiss_{\JJ}I^*(dIdv) + \poiss_{\JJ} (dIdv) J + 2\poiss_{\JJ}(L_V g) \\
&= J\poiss_{\JJ}(dJdv) + \poiss_{\JJ} (dIdv) J + 2\poiss_{\JJ}(L_V g)\\
&= -\poiss(dIdv) + I \poiss_{\JJ} (dIdv) + \poiss_{\JJ}(dIdv) J +2\poiss_{\JJ}(L_V g) \\
&= -\poiss(dIdv)  + \poiss_{\JJ}\left(-J^*(dIdv) + (dIdv) J + 2L_V g\right).  
\end{split} \]
From \eqref{I-variation}, we  thus get
\[ \poiss_{\JJ} (L_V g) = -\poiss_{\JJ} \left((d I dv) J \right)^{\rm sym},\]
and hence
\begin{equation}\label{g-variation}
 L_V g = -\left((d I dv) J \right)^{\rm sym}.
\end{equation}
Therefore,  by \eqref{I-variation} and \eqref{g-variation}, we compute
\[ \frac{d}{dt} I = - L_V I= \poiss (dIdv) \qquad \frac{d}{dt}g= - L_V g = \left((dIdv)J\right)^{\rm sym}.\]
Finally, we define a path of $2$-forms $b_t$ by 
 \[ \frac{d}{dt} b= -((dIdv)J)^{\rm skew}.\]
 This shows that $\tilde m_t := (g_t, J, I_t, b_t)$ is a smooth path in $\mc{F}^{\T}_m$,  generated by the smooth function $v_t = -(\phi_{t}^{-1})^* v$, see Lemma~\ref{lm:gc_variation}. 
\end{proof}

\begin{rmk} If we assume in Lemma~\ref{l:hredinGKclass} that the torus $\T$ is generated by the divergence $\div_{\vol_0} \pi_{\JJ_0}$ of a volume form $\vol_0$ on $M$, then letting
\[ \vol_t := (\phi_{t}^{-1})^*(\vol_0)\]
we obtain the solution of the evolution equation  for $\vol_t$ in Lemma~\ref{l:moser}. Thus, in this case, all the data $(g_t, J_t=J, I_t, \vol_t)$ in Lemma~\ref{l:moser} are induced by the flow $\phi_t$ of $V$.
\end{rmk}

\begin{thm}\label{thm:CML-reduced} Let $m=(g,I,J, 0)$ be a generically symplectic type generalized K\"ahler structure on $M$ and  $\vol$ an adapted volume form which generates a torus $\T$ of isometries of $m$.  Denote by $(\II, \JJ)$ the corresponding $(\JJ, \vol)$-adapted GK structure on $(M, H_0:=-d^c_J \omega_J)$, see Remark~\ref{r:divgremark}.  Then
\begin{enumerate}
\item $\mf{g}^{\T}(m) \cong \mf{g}^{\T}(\II, \JJ)$ is a compact Lie algebra and $\mf{g}^{\T}(m) \cap J\mf{g}^{\T}(m)=0$.
\item There exists an injective complex Lie algebra homomorphism 
\[ j: \mf{h}^{\T}_{\rm red}(m) \to  \mf{h}^{\T}(\II, \JJ), \qquad  j(\pi_{\JJ}(du) + J\pi_{\JJ}(dv)) := \Psi(\JJ du + \II \JJ dv).\]
\item Suppose $(\II, \JJ, \vol)$ is extremal and $X_{\rm ext}:= \poiss_{\JJ}(d\Gscal(\II, \JJ, \vol))=\pi_T(\chi) \in \mf{g}^{\T}(m)$ is the corresponding extremal vector field. Denote by $\hat\T$ the torus extension of $\T$ generated by $X_{\rm ext}$ and let $\mf{h}^{\hat \T}_{\rm red}(m)<\mf{h}^{\T}_{\rm red}(m)$ be the sub-algebra of $X_{\rm ext}$-invariant elements.
Then $\mf{h}^{\hat\T}_{\rm red}(m)= \mf{g}^{\T}(m)\oplus J \mf{g}^{\T}(m)$ is reductive and 
$j: \mf{h}^{\hat \T}_{\rm red}(m) \cong \left(\mf{h}^{\T}(\II, \JJ)\right)^{\chi}=\mf{g}^{\T}(\II, \JJ)\otimes \C$.
\end{enumerate}
\end{thm}
\begin{proof} 
(1) By Lemma~\ref{l:hamiltonian_Killing_converse}, $\mf{g}^{\T}(m)$ consists of \emph{all} Killing vector fields of $(g, I, J)$ of the form $V= \pi_{\JJ}(du)$, $u\in C^{\infty}_0(M, \R)^{\T}$ and the Lie algebra homomorphism $\pi_T : \mf{g}^{\T}(\II, \JJ) \to \mf{g}^{\T}(m)$ is surjective. The kernel of $\pi_T$ consists of $2$-forms $B$ such that $L_B^{\rm aut} \JJ = L_{B}^{\rm aut}\II=0$. Note that $B$ must be zero  at the points where $\pi_\JJ$ is non-degenerate, hence, by continuity, everywhere.

Next let  $V=\mf{g}^{\T}(m)\cap J\mf{g}^{\T}(m)$. Then
\[ V= \pi_{\JJ}(du)= J \pi_{\JJ}(dv)= \pi_{\JJ}(Idv).\]
By the non-degeneracy assumption for $\pi_{\JJ}$ we will have $du = I dv$
everywhere.  The latter and the maximum principle yield that $v$ and $u$ are constant, i.e. $V=0$.

\bigskip
(2) We first notice that if $V= \pi_{\JJ}(du) + J\pi_{\JJ}(dv) \in \mf{h}^{\T}_{\rm red}(m)$, then the potential functions are uniquely defined up to additive constants. Indeed, if $\pi_{\JJ}(du) + J\pi_{\JJ}(dv) =0$, then using $J \pi_{\JJ} = \pi_{\JJ} I$ and the non-degeneration of $\pi_{\JJ}$, we conclude that $du + Jdv =0$. The claim follows from the maximum principle. This shows that $j: \mf{h}^{\T}_{\rm red}(m) \to \mf{aut}(M, H_0)$ is well-defined.

Let $V= \pi_{\JJ}(du) + J\pi_{\JJ}(dv) \in \mf{h}^{\T}_{\rm red}(m)$, and consider the path $\tilde m_t=(g_t, I, I_t, b_t) \in \mc{F}^{\T}_m$ generated by the flow of $V$ as in Lemma~\ref{l:hredinGKclass}. We thus have (see the proof of Lemma~\ref{l:hredinGKclass})
\[\left. \dt \right|_{t=0} I = -L_V I, \, \, \,  \left. \dt \right|_{t=0} J = - L_V J=0, \, \,  \, \left. \dt \right|_{t=0} g= - L_V g, \, \, \, \left. \dt \right|_{t=0} b = -(dIdv)^{(2,0)+(0,2)}_J J=:B. \]
Using Theorem~\ref{thm:gualtieri_map} and $b_0=0$, we compute the derivatives at $t=0$ of $\II$ and $\JJ$ to be
\[ \left. \dt \right|_{t=0} \II =  - L_V \II + L_{B} \II, \qquad \left. \dt \right|_{t=0} \JJ = - L_V \JJ + L_{B} \JJ.\]
On the other hand,  by Lemma~\ref{lm:gc_variation}, we have
\[ \left. \dt \right|_{t=0} \II = L^{H_0}_{\II I dv} \II= L^{}_{\Psi(\II I dv)} \II, \qquad \left. \dt \right|_{t=0} \JJ = - L^{H_0}_{\JJ I dv} \JJ= - L^{}_{\Psi(\JJ I dv)} \JJ, \]
We thus derive the identities
\begin{equation}\label{basic} L^{}_{V + \Psi(\II Idv) - B} \II =0, \qquad L^{}_{V - \Psi(\JJ Idv) - B} \JJ =0.\end{equation}
Using \eqref{eq:gualtieri_map}, we have
\[ V - \JJ I dv = \frac{1}{2}\left((I+J) du^{\sharp} +dv + JIdv\right),\qquad \JJ du = \frac{1}{2}\left((I+J) du^{\sharp} + (I-J) du\right)\]
As $L^{H_0}_{\Psi(\JJ du)} \JJ=0$ (see Lemma~\ref{l:hamiltonians}) we have
\[ 0= L^{H_0}_{\Psi(V -\JJ du +\JJ Idu)-B} \JJ = \frac{1}{2}L^{H_0}_{\Psi(dv + JI dv -(I-J)du) - 2B} \JJ.\]
By the non-degeneracy condition for $\pi_{\JJ}$, we deduce
\[\frac{1}{2}\Psi(dv + JI dv -(I-J)du) = \frac{1}{2}\Psi(JI dv - (I-J)du)=B. \]
By  \eqref{eq:gualtieri_map} again, we compute
\begin{equation}\label{eq:surjective} \JJ du + \II \JJ dv = \frac{1}{2}\left( -2dv^{\sharp} + (I+J) du^{\sharp} + (I-J)du\right), \end{equation}
\[ V + \II Idv = \frac{1}{2}\left(-2dv^{\sharp} + (I+J) du^{\sharp} -dv +JIdv\right).\]
It follows that
\[ \Psi(V + \II Idv)-\Psi(\JJ du + \II\JJ dv)= \frac{1}{2}\Psi(-dv +JIdv -(I-J)du)= \frac{1}{2} \Psi(JIdv -(I-J)du) = B.\]
Substituting back in \eqref{basic}, we get
\[ L_{\Psi(\JJ du + \II \JJ dv)} \II =0.\]
This shows that $j: \mf{h}^{\T}_{\rm red}(m) \to \mf{h}^{\T}(\II, \JJ)$ is a complex-linear map; it is a Lie algebra homomorphism by Lemma~\ref{l:hamiltonians} (and the fact that $\mf{ham}^{\T}(\JJ)$ acts holomorphically on $\mc{AGK}(\JJ, \vol)$) and the formula for the Lie bracket on $\mf{h}^{\T}_{\rm red}(m)$ in the proof of Lemma~\ref{l:reduced}. The injectivity of $j$ follows from \eqref{eq:surjective} and the maximum principle.

\bigskip 

(3) First, by the arguments in the proof of Lemma~\ref{l:hamiltonian_Killing_converse}, we have that
\[\Psi(\JJ d\Gscal(\II, \JJ, \vol)) = X_{\rm ext},\]
showing that 
\[ \left(\mf{h}^{\T}(\II, \JJ)\right)^{\chi}=\{ \Psi(\JJ du) + \i \Psi(\JJ dv) \in \mf{h}^{\T}(\II, \JJ) \, | \, L_{X_{\rm ext}} u= L_{X_{\rm ext}} v=0\}.\]
It thus follows that
\[ j: \mf{h}^{\hat \T}_{\rm red}(m) \to \left(\mf{h}^{\T}(\II, \JJ)\right)^{\chi}.\]
Also, as $\Gscal(\II, \JJ, \vol)$ is a smooth function depending only on $(g, I, J, \vol)$ (see Theorem~\ref{thm:scalformula} below), $\Gscal(\II, \JJ, \vol)$  is invariant by any element $V\in \mf{g}^{\T}(m)$ (recall $V$ preserves $\vol$ by Lemma~\ref{l:hamiltonians} and  $(g, I, J)$ by definition). In particular,
\[ \mf{g}^{\T}(m) \oplus J \mf{g}^{\T}(m) < \mf{h}^{\hat \T}_{\rm red}(m).\]
By Lemma~\ref{l:hamiltonian_Killing_converse}, $j: \mf{g}^{\T}(m) \cong \mf{g}^{\T}(\II, \JJ)$, hence also (examining the definition of $j$)
\[ j: \mf{g}^{\T}(m) \oplus J \mf{g}^{\T}(m) \cong \mf{g}^{\T}(\II, \JJ)\otimes \C. \]
By Theorem~\ref{thm:CML}, we know that
\[ \mf{g}^{\T}(\II, \JJ)\otimes \C =\left(\mf{h}^{\T}(\II, \JJ)\right)^{\chi}\]
By the injectivity of $j$ we conclude $\mf{g}^{\T}(m) \oplus J \mf{g}^{\T}(m) = \mf{h}^{\hat \T}_{\rm red}(m)  \overset{j}{\cong}\left(\mf{h}^{\T}(\II, \JJ)\right)^{\chi}$.
\end{proof} }

\begin{rmk}\label{r:comments-reduced} (1) Presumably, $\mf{h}_{\rm red}^{\hat \T}(m)$ should be independent of the choice of $m \in \mc{F}^{\T}$ (this is established in the symplectic type case in \cite{ASU2}) and is thus a Lie algebra that can be computed \emph{a priori}. Theorem~\ref{thm:CML-reduced} therefore can be used to determine the infinitesimal symmetries of an extremal bihermitian data $(g, I, J,\vol)$ in $\mc{F}^{\T}$. 

(2) In the K\"ahler case, $\mf{h}_{\rm red}^{\T}(m)$ is the Lie algebra of holomorphic vector fields on $(M,J)$ with non-empty zero set (see e.g. \cite{gauduchon-book}). Similarly, in the symplectic type GK case, the Lie algebra $\mf{h}_{\rm red}^{\T}(m)$ is identified in \cite{ASU2} with the sub-algebra of $\mf{aut}(M, \pi, J)$ of holomorphic vector fields preserving the holomorphic Poisson tensor $\pi - \sqrt{-1}J\pi$ on $(M, J)$ and  which couple trivially with any closed $(1,0)$-form on $(M, J)$.

(3)  There is a natural complex Lie algebra isomorphism $q: \mf{h}^{\T}_{\rm red}(m) \to \mf{h}^{\T}_{\rm red}(\bar m)$, where $\bar m=(g, J, I, b)$ denotes the bi-Hermitian data with $I$ and $J$ exchanged. Specifically,
\[ q(\pi_{\JJ} (du) + J \pi_{\JJ}(dv)):=\pi_{\JJ} (du) + I \pi_{\JJ}(dv). \]
The fact that $q(V)$ preserves $I$ follows by the expression $q(V)=V + \pi(dv)$, the formula $L_{V} I = -\pi (dd^c_I v)$ established in the proof of Lemma~\ref{l:reduced},  and the general fact $L_{\pi(dv)} I= \pi (dd^c_I u)$ (which holds for any smooth function $v$ by using that $\pi -\sqrt{-1}I\pi$ is a holomorphic Poisson tensor on $(M, I)$).  
\end{rmk}
\begin{ex}[Diagonal Hopf surfaces] Let $(M, J)= (\C^{2}\setminus \{0\})/\langle A \rangle$ be a diagonal Hopf complex surface, where $A={\rm diag}(\alpha, \beta), \, 0<|\alpha|<1, 0<|\beta|<1$ is a contraction of $\C^2$. It is known~\cite{SU2, ASU3} that $(M,J)$ admits a generalized K\"ahler Ricci soliton $m=(g, J, I,f)$ such that $\pi-\sqrt{-1} J\pi$ is a holomorphic section of the anti-canonical bundle vanishing along two disjoint smooth elliptic curves $E_1 \cup E_2$. In particular, $\pi_{\JJ}$ is non-degenerate on the open dense subset $M\setminus(E_1\cup E_2)$. By Corollary~\ref{c:GITforGKRS}, this is an instance of a cscGK structure with respect to the volume form $\vol=e^{-f} dV_g$, where $f$ is the soliton potential. The connected Lie group ${\rm Aut}_0(M, \pi, J)$ is (see \cite[Prop.6.3]{ASU3}) $\left(\C^{*} \times \C^{*}\right)/\langle A \rangle$, showing that $\mf{aut}(M, \pi, J)= \C^2$. We also notice that ${\rm Aut}_0(M, \pi, J)$ admits a $3$-dimensional \emph{maximal} compact torus $K$, generated by the $A$-invariant vector fields
\[\left\{\im(z_1 \frac{\partial}{\partial z_1}),  \im(z_2 \frac{\partial}{\partial z_2}), \log|\alpha|{\rm Re}(z_1 \frac{\partial}{\partial z_1}) +\log|\beta|{\rm Re}(z_2 \frac{\partial}{\partial z_2})\right\}. \]
By Lemma~\ref{l:solitonID} the divergence $\div_{\vol} \pi_{\JJ}$ is $X=X_I + X_J$, which is identified in the proof of Theorem~6.6 in \cite{ASU3} with a real multiple of $\im(z_2 \frac{\partial}{\partial z_2})$. It thus follows that in this case $\T=S^1$ corresponds to rotation in the $z_2$-axis of $\C^2$. This proof also shows that
the commuting Killing fields $(X_I, X_J)$ define a 2-dimensional torus in $\Aut_0(M, \pi, J)$,  generated by $\{ \im(z_1 \frac{\partial}{\partial z_1}),  \im(z_2 \frac{\partial}{\partial z_2})\}$. By Lemma~6.8 in \cite{ASU3},
\[ \mf{h}^K_{\rm red}(m) =\{ 0\}.\]
As $\mf{aut}(M, \pi, J)$ is abelian, we have for any torus $\T \subset \rm{Aut}_0(M, \pi, J)$
\[ \mf{h}_{\rm red}^{\T}(m) = \mf{h}^{K}_{\rm red}(m)=\{ 0\}.\]
\end{ex}

\subsection{Futaki invariant}

Here we introduce a generalized K\"ahler analogue of the Futaki character, extending the definition in the symplectic-type case given in \cite{ASUScal}.   We are in the setup of Assumption~\ref{a:generically-symplectic}.

\begin{defn} Denote by  $\hred^{\T}(J)$ the Lie algebra of all $\T$-invariant real holomorphic vector fields $V$ on $(M, J)$,  such that for each $m\in \mc{F}^{\T}$, there exist (unique by the generically symplectic type assumption) smooth functions $u_m, v_m \in C_0^{\infty}(M, \vol_m)^{\T}$  with
\begin{equation}\label{eq:reduced-potential} V= \pi_{\JJ_m}(du_m) + J \pi_{\JJ_m}(dv_m).  \end{equation}
Observe that by the results in Section~\ref{s:red} and Lemma~\ref{l:moser}, we have
\[ \hred^{\T}(J) = \bigcap_{m \in \mc{F}^{\T}} \mf{h}_{\rm red}^{\T}(m),\]
showing that $\hred^{\T}(J)$ is a complex sub-algebra of $\mf{aut}(M, \poiss, J)$.
\end{defn}

\begin{thm}[Futaki character]\label{t:futaki-character}
    Let $m \in\mathcal{F}^{\T}$ be a generalized K\"ahler structure.  Define a homomorphism ${\bf F}_{m}\colon \hred^{\T}(J)\to \R$ by
    \begin{equation}\label{e:futaki-character}
        {\bf F}_{m}(V):=\int_M v_{m}\Gscal(m, \vol_m)\vol_m
        =-\boldsymbol{\tau}(\mathbf X_{v_{m}}).
    \end{equation}
    Then ${\bf F}_{m}$ is independent of $m\in\mc{F}^{\T}$ and vanishes on the commutator $[\hred^{\T}(J),\hred^{\T}(J)]$.
    In particular,  ${\bf F}={\bf F}_{m}$ is a character of $\hred^{\T}(J)$ which is identically zero if $\mc{F}^{\T}$ admits a cscGK metric.
\end{thm}

\begin{proof}
The infinitesimal action of the vector field $V\in\hred^{\T}(J)$ on $m\in\mc{F}^{\T}$ induces a vector field ${\mathbf X}$ on $\mc{F}^{\T}$ by Lemma \ref{l:hredinGKclass}, which at a point $m\in\mc{F}^{\T}$ is given by
\[
{\mathbf X}_m={\mathbf X}_{v_{m}},
\]
where $v_m$ is the unique normalized potential for $V$ associated to $m \in \mc{F}^{\T}$ via \eqref{eq:reduced-potential}.
Now, the 1-form $\boldsymbol{\tau}$ on $\mc{F}^{\T}$ is invariant under the vector field $\mathbf{X}$, since it is induced by an action of $\mathrm{Diff}(M)$ on the entire structure $(g, I, J, \vol)$ as explained in Lemma \ref{l:hredinGKclass}.  Thus by Cartan's formula we have
\[
0=d(\boldsymbol{\tau}(\mathbf X))+d\boldsymbol\tau({\mathbf X},\cdot).
\]
Since $\boldsymbol{\tau}$ is closed, it follows that $\boldsymbol{\tau}(\mathbf X)$ is constant on $\mc{F}^{\T}$ which is equivalent to the statement of the constancy of ${\bf F}_{m}(V)$ on $\mc{F}^{\T}$.

Furthermore, if $W\in\hred^{\T}(J)$ is another holomorphic vector field inducing a vector field $\mathbf Y$ on $\GK_0^{\T}$, then
\[
{\bf F}_{m}([V,W])=-\boldsymbol\tau([\mathbf X,\mathbf Y])=d\boldsymbol{\tau}(\mathbf X,\mathbf Y)+\mathbf{Y}\cdot \boldsymbol{\tau}(\mathbf X)-\mathbf{X}\cdot \boldsymbol{\tau}(\mathbf Y)=0,
\]
as claimed, so that ${\bf F}_{m}$ is a character of $\hred^{\T}(J)$.
\end{proof}

\begin{rmk} A conceptually distinct notion of Futaki invariant appeared recently in the context of generalized geometry \cite{garcia2023futaki}, yielding obstructions to the existence of solutions of the Hull-Strominger system.
\end{rmk}

\appendix
\section{The general Calabi-Lichnerowiz-Matsushima theorem}\label{s:CLM} 

In this appendix we show a general Calabi-Lichnerowicz-Matsuhima result which follows from the arguments of \cite{LSW}, but which is not stated there exactly in the generality below.

\subsection{The formal momentum map setup}

The formal setup is as follows:
\begin{enumerate}
   \item[$\bullet$] $(\bf{M}, \bf{J}, \boldsymbol{\Omega}, \bf{g})$ is a formal (possibly infinite dimensional) Fr\'echet almost Hermitian manifold, with tangent space ${\bf T}_m{\bf M}$. By this, we mean that each tangent vector  ${\bf X}(m) \in {\bf T}_m{\bf M}$ can be obtained
   as the velocity  at $t=0$ of a differentiable curve $m(t) \in {\bf M}$ with $m(0)=m$. We denote by $\boldsymbol{\mf{X}}({\bf M})$ the space of (formal) vector fields on ${\bf M}$. 
   \item[$\bullet$] $\mf{Aut}$ is a (possibly infinite dimensional) Lie group which acts differentiably (on the left)  
   on ${\bf M}$ and $\mathfrak{ham}$ is a (possibly infinite dimensional) Lie algebra admitting an exponential map ${\rm exp} : \mf{aut} \to \mf{Aut}$; this gives rise to fundamental vector fields on ${\bf M}$ which admit flows:
   \[ \mf{ham} \ni u \mapsto {\bf X}_u \in \boldsymbol{\mf{X}}({\bf M}), \qquad {\bf X}_u(m) := -\frac{d}{dt}_{|_{t=0}} {\rm exp}(tu) \cdot m. \]  
   \item[$\bullet$] The representation  $\mf{ham} \ni u \mapsto {\bf X}_u \in \boldsymbol{\mf{X}}({\bf M)}$ coming from the $\mf{Aut}$-action is  (formally) hamiltonian with respect to ${\bf \Omega}$ in the sense that there exists a differentiable  map ${\boldsymbol{\mu}}: {\bf M} \to \mf{ham}^*$
   such that for any $u, v \in \mf{ham}$, any $m \in {\bf M}$ and any ${\bf X}(m) \in {\bf T}_m {\bf M}$, the corresponding fundamental vector fields ${\bf X}_u, {\bf X}_v$ on ${\bf M}$ satisfy 
   \begin{equation}\label{momentum-property} \boldsymbol{\Omega}_m({\bf X}_u, {\bf X}(m)) = - d\left(\langle u, \boldsymbol{\mu}\rangle\right)({\bf X}(m)) = - \frac{d}{dt}_{|_{t=0}}\langle u, \boldsymbol{\mu}\rangle(m_t),\end{equation}
   where $m_t$ is a curve through $m$ with velocity ${\bf X}(m)$, and
   \begin{equation}\label{equivariant-property}\boldsymbol{\Omega}_m({\bf X}_u, {\bf X}_v) = -\langle [u,v], \boldsymbol{\mu}(m)\rangle.\end{equation}
   \item[$\bullet$] $\mf{ham}$ is   endowed with an ad-invariant inner product $\llangle \cdot ,  \cdot \rrangle$ and  $\boldsymbol{\mu}$ admits a  differentiable $\llangle \cdot, \cdot \rrangle$-dual map  $\boldsymbol{\mu}^{\sharp}: {\bf M} \to \mf{ham}$,  such that for any $m\in {\bf M}$  and any $u \in \mf{ham}$
   \begin{equation}\label{dual-property}\langle \boldsymbol{\mu}(m), u \rangle = \llangle \boldsymbol{\mu}^{\sharp}(m), u \rrangle. \end{equation}
\end{enumerate}
The example of interest in this paper is when ${\bf M}=\mc{AGK}(\JJ, \vol)$  and $\mf{Aut}=\Aut(M, H_0, \JJ, \div_{\vol}(\JJ))$ is the group of $\T$-invariant generalized automorphisms of $(E=TM \oplus T^*M, H_0)$ preserving  $\JJ$ and $\div_{\vol}(\JJ)$. Note that $\mf{Aut}$ has a Lie algebra $\mf{aut}(M, H_0, \JJ, \div_{\vol}(\JJ))$ with exponential map coming from the flow of vector fields  composed with the $B$-field action. Furthermore, in our case $\mf{ham}= \mf{ham}^{\T}(\JJ)= \{\Psi(\JJ du), \, u\in C^{\infty}_0(M, \R)^{\T}\} < \mf{aut}(M, H_0, \JJ, \div_{\vol}(\JJ))$ and the ad-invariant inner product $\llangle \cdot, \cdot \rrangle$ correspond to the global $L^2$-product  on $C^{\infty}_0(M, \R)^{\T}$ with respect to $\vol$.  We then have $\boldsymbol{\mu}^{\sharp}(m)=\Psi(\JJ d\mathring{\Gscal}(\II, \JJ, \vol))$ where $\mathring{\Gscal}(\II, \JJ, \vol)$ is the normalized Goto scalar curvature in $C^{\infty}_0(M, \vol)$.

\subsection{Extremal points}
\begin{defn} An \emph{extremal point} $m\in {\bf M}$ is a critical point of the square-norm function 
\[ m \to || \boldsymbol{\mu}^{\sharp}(m)||^2.\]
\end{defn}
For any given point $m\in {\bf M}$, we denote in the sequel by $\mf{g}(m)\subset \mf{ham}$ the  ``stabilizer'' of  $m$, i.e.
\[ \mf{g}(m) :=\{ u \in \mf{ham} \, | \, {\bf X}_u(m) =0\}.\]
In general, $\mf{g}(m)$ is a vector subspace of $\mf{ham}$, but in many cases it will also be a subalgebra. This is the case for the example we are interested in,  as we can then identify $\mf{g}(m)$ with the subagebra of hamiltonian elements $\Psi(\JJ du)$ in $\mf{aut}(M, H, \JJ, \vol)$ whose exponential flow preserves $\II$.
\begin{lemma}\label{l:extremal} $m\in {\bf M}$ is extremal if and only if
\[ \chi:=\boldsymbol{\mu}^{\sharp}(m) \in \mf{g}(m).\]
Furthermore, in this case,  $[\chi, u]=0$ for any   $u \in \mf{g}(m)$.
\end{lemma}
\begin{proof}  We denote by $m_t$ a smooth curve through $m$ with velocity ${\bf X}(m) \in {\bf T}_m{\bf M}$.
 We thus have (using \eqref{dual-property} and \eqref{momentum-property}):
\[\frac{1}{2}\left. \dt \right|_{t=0} ||\boldsymbol{\mu}^{\sharp}(m_t)||^2 = \left. \dt \right|_{t=0} \llangle \chi,  \boldsymbol{\mu}^{\sharp}(m_t) \rrangle = d\langle \chi, \boldsymbol{\mu}\rangle({\bf X}(m))=-\boldsymbol{\Omega}_m({\bf X}_{\chi}(m), {\bf X}(m)).\]
 The first claim follows as $\boldsymbol{\Omega}_m$ is non-degenerate (being the fundamental form of an almost Hermitian structure). To show that $\chi$ is in the center, we use  \eqref{equivariant-property} and the ad-invariance of $\llangle \cdot, \cdot \rrangle$: for any $u\in \mf{g}(m)$ and $v\in \mf{ham}$, 
 \[0=\boldsymbol{\Omega}_m({\bf X}_u(m), {\bf X}_v(m))= -\langle [u,v], \boldsymbol{\mu}(m)\rangle =-\llangle [u,v], \chi \rrangle= \llangle v, [u, \chi]\rrangle.\]
\end{proof}

\subsection{The Calabi--Licherwowicz-Matsushima theorem}

We define a map
\[ \mf{ham}\otimes \C \ni u+\sqrt{-1}v \to {\bf X}_u + {\bf J} {\bf X}_v \in \boldsymbol{\mf{X}}({\bf M}).\]
This is a complex-linear map when we endow $\boldsymbol{\mf{X}}({\bf M})$ with the almost complex structure ${\bf J}$.  If one assumes that ${\bf J}$ is formally integrable and the
infinitesimal action of $\mf{ham}$  is ${\bf J}$-holomorphic,  the above map gives rise to a $\C$-linear representation of the complex Lie algebra $\mf{ham}\otimes \C$ into $\boldsymbol{\mf{X}}({\bf M})$. This is sometimes referred to as  the \emph{complexified infinitesimal action} of $\mf{ham}\otimes\C$ on ${\bf M}$, although we emphasize here that we do not assume this property.

We next fix an extremal point $m\in {\bf M}$  and consider the complex vector sub-space
\[\mf{h}(m):=\{ u + \sqrt{-1}v \in \mf{ham}\otimes \C \, | \, {\bf X}_u(m) + {\bf J} {\bf X}_v(m)=0\} \subset \mf{ham}\otimes \C. \]

Clearly, we have an inclusion of complex vector spaces
\[ \mf{g}(m)\otimes \C \subset \mf{h}(m).\]

\begin{thm}[Calabi-Lichnerowicz-Matsushima]\label{thm:CML-decomp} In the setup above, let
\[ (\mf{h}(m))^{\chi}:= \{ u + \sqrt{-1}v \in \mf{h}(m) \, \,  | \, \, [\chi, u]=[\chi, v]=0\}. \]
Then $\mf{g}\otimes \C= (\mf{h}(m))^{\chi}$.
\end{thm}
\begin{proof} We know by Lemma~\ref{l:extremal} that $\mf{g}(m)\otimes \C \subset (\mf{h}(m))^{\chi}$. We want to show equality. Let $u+\sqrt{-1}v$ be an element in $(\mf{h}(m))^{\chi}$, i.e.
\[ [\chi, u] = 0, \qquad [\chi, v] =0.\]
It then follows by the ad-invariance of $\llangle \cdot, \cdot \rrangle$, \eqref{equivariant-property} and ${\bf X}_u(m)=-{\bf J} {\bf X}_v(M)$ (by the  definition of $\mf{g}(m)$):
\[
\begin{split}
0 &= \llangle [\chi, u], v \rrangle = \llangle \chi, [u,v]\rrangle \\
&= \langle \boldsymbol{\mu}(m), [u,v] \rangle =-\boldsymbol{\Omega}_{m}({\bf X}_{u}(m), {\bf X}_v(m)) \\
&= \boldsymbol{\Omega}_{m}({\bf J}{\bf X}_{v}(m), {\bf X}_v(m))=-{\bf g}({\bf X}_{v}(m), {\bf X}_{v}(m)).
\end{split}\]
We thus obtain that ${\bf X}_v(m)={\bf X}_u(m)=0$, i.e.  $u + \sqrt{-1}v \in \mf{g}(m)\otimes \C$.
\end{proof}
\begin{rmk} By a similar argument one shows that if $\sqrt{-1}\lambda$ is an eigenvalue of the  skew-Hermitian endomorphism ${\rm ad}_{\chi}$ of $(\mf{h}(m), \llangle \cdot, \cdot \rrangle$), then $\lambda \geq 0$. In the case when $\mf{h}(m)$ is a finite dimensional complex Lie algebra, we thus get the 
Calabi decomposition  of $\mf{h}(m)$ as a direct sum of ${\rm ad}_{\chi}$ eigenspaces:
\[ \mf{h}(m)= (\mf{h}(m))^{\chi} \oplus \bigoplus_{\lambda >0} (\mf{h}(m))_{\lambda}. \] The above  decomposition yields that  $\mf{g}(m)$ is a maximal (real) compact subalgebra of $\mf{h}(m)$, see e.g. \cite{LSW} for details.
\end{rmk}

\end{document}